\documentclass[a4paper]{amsart}

\usepackage{latexsym,amsfonts,amsthm,amsmath,amscd,amssymb}
\usepackage{graphicx}

\newtheorem{lemma}{Lemma}[section]
\newtheorem{thm}[lemma]{Theorem}
\newtheorem{rem}[lemma]{Remark}
\newtheorem{prop}[lemma]{Proposition}
\newtheorem{cor}[lemma]{Corollary}
\newtheorem{conj}[lemma]{Conjecture}

\newtheorem{example}[lemma]{Example}
\newtheorem{defn}[lemma]{Definition}

\def\H{{\mathbb{H}}}
\def\C{{\mathbb{C}}}
\def\Q{{\mathbb{Q}}}

\def\bd{{\partial}}
\newcommand\matT{{\mathbb{T}}}

\newcommand\matP{{\mathbb{P}}}

\newcommand\matZ{{\mathbb{Z}}}
\newcommand\matK{{\mathbb{K}}}

\newfont{\Got}{eufm10 scaled 1200}

\newcommand{\mycap}[1]{\caption{{#1}}}

\newcommand{\calT}{{\mathcal T}}

\title{Hyperbolic graphs of small complexity}

\author[Heard]{Damian Heard}
\address{Department of Mathematics and Statistics,
University of Melbourne, Parkville, Victoria 3010, Australia}
\curraddr{RedTribe, Level 10, 50 Market St, Melbourne, Victoria 3000, Australia}
\email{damian.heard at gmail.com}

\author[Hodgson]{Craig Hodgson}
\address{Department of Mathematics and Statistics,
University of Melbourne, Parkville, Victoria 3010, Australia}
\email{cdh at ms.unimelb.edu.au}
\thanks{The research of the first two authors was partially
supported by the ARC grant DP0663399; that of the last two authors
by the INTAS project ``CalcoMet-GT'' 03-51-3663.}

\author[Martelli]{Bruno Martelli}
\address{Dipartimento di Matematica ``Tonelli'', Largo Pontecorvo 5, 56127 Pisa, Italy}
\email{martelli at dm.unipi.it}

\author[Petronio]{Carlo Petronio}
\address{Dipartimento di Matematica Applicata ``Dini'', Via Buonarroti 1/C, 56127 Pisa, Italy}
\email{petronio at dm.unipi.it}

\begin{document}

\begin{abstract}
\noindent
In this paper we enumerate and classify the ``simplest''
pairs $(M,G)$ where $M$ is a closed orientable $3$-manifold and $G$
is a trivalent graph embedded in $M$.

To enumerate the pairs we use a variation of Matveev's definition of
complexity for $3$-manifolds, and we consider only
$(0,1,2)$-irreducible pairs, namely pairs $(M,G)$ such that any
2-sphere in $M$ intersecting $G$ transversely in at most $2$ points bounds a ball
in $M$ either disjoint from $G$ or intersecting $G$ in an unknotted
arc. To classify the pairs our main tools are geometric invariants
defined using hyperbolic geometry. In most cases, the graph
complement admits a unique hyperbolic structure {\em with parabolic
meridians}; this structure was computed and studied using Heard's
program {\em Orb} and Goodman's program {\em Snap}.

We determine all $(0,1,2)$-irreducible pairs up to complexity 5,
allowing disconnected graphs but forbidding components without
vertices in complexity 5. The result is a list of 129 pairs, of
which 123 are hyperbolic with parabolic meridians. For these pairs
we give detailed information on hyperbolic invariants including
volumes, symmetry groups and arithmetic invariants.   Pictures of
all hyperbolic graphs up to complexity 4 are provided.  We also
include a partial analysis of knots and links.

The theoretical framework underlying the paper is twofold, being based on
Matveev's theory of spines and on Thurston's idea (later developed by
several authors) of constructing hyperbolic structures via triangulations.
Many of our results were obtained (or suggested) by computer investigations.
\end{abstract}

\subjclass[2000]{Primary 57M50;  Secondary 57M27, 05C30, 57M20.}

\maketitle

\section{Introduction}

The study of {\em knotted graphs}  in $3$-manifolds is a natural
generalization of classical knot theory, with potential applications
to chemistry and biology (see \emph{e.g.}~\cite{Fla}). In knot
theory, extensive knot tables have been built up through the work of
many mathematicians (see \emph{e.g.} Conway~\cite{Con} and
Hoste--Thistlethwaite--Weeks~\cite{HTW}). There has been much less
work on the tabulation of knotted graphs,  but some knotted graphs
in $S^3$ have been enumerated in order of crossing number by
Simon~\cite{Si}, Litherland~\cite{Li}, Moriuchi~\cite{Mori,Mori1}, and
Chiodo~et.~al.~\cite{CHHSS}.

In this paper we classify the simplest
trivalent graphs in general closed $3$-manifolds. We first enumerate them
using a notion of complexity which extends Matveev's definition for
$3$-manifolds~\cite{Matv:AAM}, and then we classify them with
the help of geometric invariants, mostly defined using hyperbolic geometry.

More precisely, the objects considered in this paper are pairs
$(M,G)$ where $M$ is a closed, connected  orientable $3$-manifold
and $G$ is a trivalent graph in $M$. The graph $G$ may contain loops
and multiple edges, and is possibly disconnected (in particular, $G$ can be a knot or a link). To avoid ``wild'' embeddings we work in the piecewise linear
category: thus $M$ is a PL-manifold and $G$ is a 1-dimensional
subcomplex, and we aim to classify graphs up to PL-homeomorphisms of
pairs.

\bigskip

Following~\cite{Matv:AAM}, a compact polyhedron $P$ is called \emph{simple}
if the link of every point of $P$ embeds in the 1-skeleton of the tetrahedron
(the complete graph with 4 vertices). Points having the whole of this graph
as a link are called \emph{vertices} of $P$. Moreover,
as defined in~\cite{orb:compl}, $P$ is a
\emph{spine} of a pair $(M,G)$ if it embeds in $M$ so that its
complement is a finite union of
balls intersecting $G$ in the simplest possible ways, as shown in
Figure~\ref{balls:fig}.
    \begin{figure}
    \begin{center}
    \includegraphics[scale=1.0]{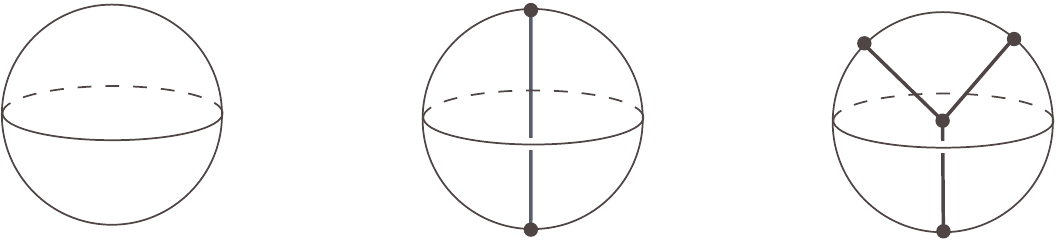}
    \mycap{\label{balls:fig} Balls in the complement of a spine.}
    \end{center}
    \end{figure}
As usual in complexity theory, the complexity $c(M,G)$ is then defined
as the minimal number of vertices in a simple spine of $(M,G)$.
The case considered in~\cite{orb:compl} is actually that
of 3-orbifolds, but the definition of complexity is the same as
just given, except that a contribution of the edge labels is also introduced.
When $G=\emptyset$ we recover the original definition of
Matveev, thus getting the equality $c(M) = c(M,\emptyset)$. In
general, we have $c(M)\leqslant c(M,G)$.

For manifolds, Matveev showed that complexity is additive under connected
sum and that it behaves particularly well on \emph{irreducible}
manifolds (\emph{i.e.}~manifolds in which every 2-sphere bounds a 3-ball).
In particular, there exist only finitely many irreducible manifolds
with given complexity. These facts extend to the context of the pairs $(M,G)$ described above, with the following notion of irreducibility:
$(M,G)$ is  \emph{$(0,1,2)$-irreducible} if
every 2-sphere embedded in $M$ and meeting $G$ transversely in at
most two points bounds a ball intersecting $G$ as in
Figure~\ref{balls:fig}, left or centre (in particular, there exists
no 2-sphere meeting $G$ in one point).

\bigskip

This paper is devoted to the enumeration and the geometric investigation of
all $(0,1,2)$-irreducible graphs $(M,G)$ of small
complexity. As usual in $3$-dimensional topology, a key role in
the study of our graphs is played by invariants coming from hyperbolic
geometry, which in particular provided the tools we used
in most cases to distinguish the pairs
from each other.

While the complement of $G$ in $M$ very often has no hyperbolic structure
with geodesic boundary (for instance, it is often a handlebody),
most pairs $(M,G)$ are indeed hyperbolic in a more general sense, namely
they are \emph{hyperbolic with parabolic meridians}. This means that
$M\setminus G$ carries a metric of constant sectional curvature $-1$ which
completes to a manifold with non-compact geodesic boundary having:
\begin{itemize}
\item toric cusps at the knot components of $G$,
\item annular cusps at the meridians of the edges of $G$, and
\item geodesic 3-punctured boundary spheres at the vertices of $G$.
\end{itemize}
This hyperbolic structure is the natural analogue of the complete
hyperbolic structure on a knot or link complement and is also
useful when studying orbifold structures on $(M,G)$.

By Mostow-Prasad rigidity, a hyperbolic structure with parabolic meridians
is unique if it exists,
so its geometric invariants only depend on $(M,G)$.
One can therefore use the volume and Kojima's canonical decomposition~\cite{Koj1,Koj2}
to distinguish hyperbolic graphs. For the pairs in our list we have
constructed and analyzed the hyperbolic structure
using the computer program \emph{Orb},
written by the first named author~\cite{orb}.

Since knots and links have already been widely studied in many contexts,
this paper focuses mostly on graphs containing vertices.

\subsection*{Number of hyperbolic graphs}
Table~\ref{hyp:num:tab} gives a summary of our results. Up to
complexity $4$ our census of hyperbolic graphs $(M,G)$ is complete
and contains $45$ elements, consisting of 5 knots, 24
$\theta$-graphs, 13 handcuffs, and 3 distinct connected graphs
with four vertices. The graph types occurring are shown in Figure~\ref{g:types:fig}.
In complexity 5 we decided to rule out
knot components, and we found 78 more hyperbolic graphs. Out
of our 123 graphs, 36 lie in $S^3$.

\begin{table}
\begin{center}
\begin{tabular}{l||r|r|r|r|r}
type & $c=1$ & $c=2$ & $c=3$ & $c=4$ & $c=5$  \\ \hline\hline
knot (in $S^3$) & 0 (0) & 0 (0) & 1 (1) & 4 (1) & -- (--) \\ \hline
$2t$ (in $S^3$) & 0 (0) & 2 (1) & 4 (1) & 18 (4) & 49 (10) \\ \hline
$2h$ (in $S^3$) & 1 (1) & 1 (0) & 3 (2) & 8 (2) & 27 (8) \\ \hline
$4a$ (in $S^3$) & 0 (0) & 1 (1) & 0 (0) & 0 (0) & 2 (2) \\ \hline
$4b$ (in $S^3$) & 0 (0) & 0 (0) & 0 (0) & 1 (1) & 0 (0) \\ \hline
$4c$ (in $S^3$) & 0 (0) & 0 (0) & 0 (0) & 1 (1) & 0 (0) \\
\end{tabular}
\end{center}
\mycap{\textbf{Numbers of hyperbolic graphs.}  When $c=5$ we have not
investigated graphs having knot components. The other graph types not mentioned
were all investigated and found to have no representative.\label{hyp:num:tab}}
\end{table}

    \begin{figure}
    \begin{center}
    \includegraphics[scale=1.0]{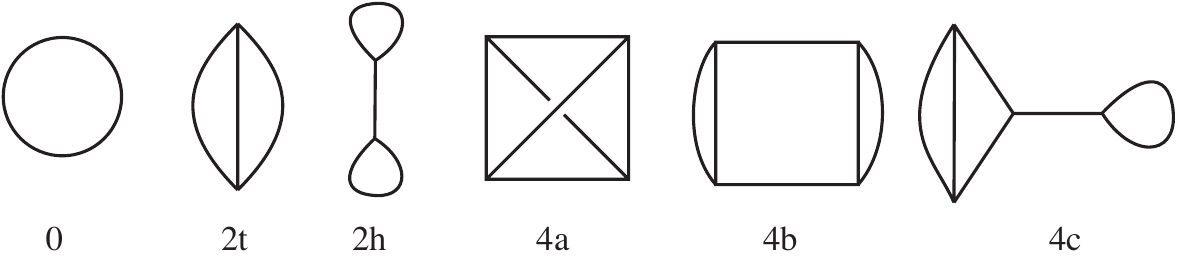}
    \mycap{Names of abstract graph types.\label{g:types:fig}}
    \end{center}
    \end{figure}

Detailed information on all the 123 hyperbolic graphs up to
complexity 5, including the volume and a description of the
canonical decomposition, will be given in
Section~\ref{res:section}, while pictures of graphs up to
complexity 4 will be shown in Section~\ref{fig:section}.

\subsection*{Complexity and volume}
As shown in
Table~\ref{hyp:num:tab}, there is a single hyperbolic graph of
smallest complexity $c=1$. It is a handcuff graph in $S^3$, described in
Figure~\ref{handcuffs} and Example~\ref{handcuff:example}. It is
also the hyperbolic graph with vertices of least volume
3.663862377... This fact confirms the following relationships
between complexity and hyperbolic geometry, which have  already been
verified for closed manifolds~\cite{Matv:AAM, MOM2}, 
cusped manifolds~\cite{CaHiWe, CaMe, MOM1,MOM2}, 
and manifolds with arbitrary (geodesic) boundary~\cite{Fu, KoMi, Mi, FriMaPe}:

\begin{enumerate}
\item
Objects having complexity zero are not hyperbolic.
\item
Among hyperbolic ones, the objects having lowest volume have the lowest complexity.
\end{enumerate}
Note that complexity and volume may share the same first segments of
hyperbolic objects (as they do) but are qualitatively different
globally, because in general there are finitely many hyperbolic
objects of bounded complexity, while infinitely many ones may have
bounded volume thanks to Dehn surgery.

\subsection*{Compact totally geodesic boundary}
It may happen that $M\setminus G$ has a hyperbolic metric which
completes to a manifold with \emph{compact} totally geodesic
boundary. In this case we say that $(M,G)$ is
\emph{hyperbolic with geodesic boundary}, which implies that
$(M,G)$ is also hyperbolic (with parabolic meridians), but as mentioned
above the converse is often false. By analyzing the graphs in
Table~\ref{hyp:num:tab}, we have established the following:

\begin{prop}\label{cpt:bd:prop}
Up to complexity $5$ there exist
$3$ graphs $(M,G)$ which are hyperbolic with geodesic boundary, shown in
Figure~\ref{geod_bound:fig}. They
all belong to the set of $8$ minimal-volume such manifolds
described by Kojima--Miyamoto~\cite{KoMi} and Fujii~\cite{Fu}, 
and they include Thurston's knotted $Y$~\cite{Th:book}.
\end{prop}

    \begin{figure}[h]
    \begin{center}
    \includegraphics[]{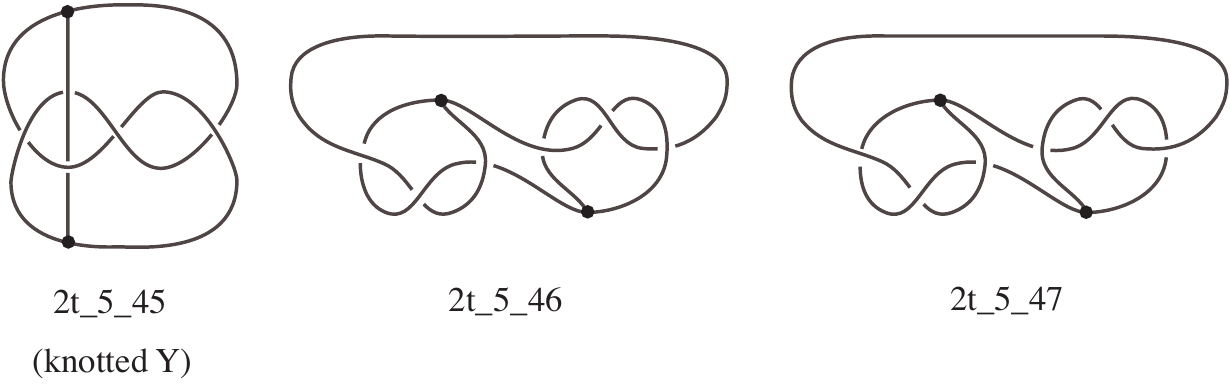}
    \mycap{\label{geod_bound:fig} Graphs whose complements admit a hyperbolic structure with geodesic boundary.}
    \end{center}
    \end{figure}

\subsection*{Non-hyperbolic graphs}
The (0,1,2)-irreducible graphs of complexity 0 were detected by
theoretical means, see Section~\ref{compl:section}. There are 3
knots (cores of Heegaard tori in $S^3$, $L(3,1)$ and $\matP^3$)
and the trivial $\theta$-graph in $S^3$, and they are all
non-hyperbolic. In complexity $c=1,2$ we have classified all
(0,1,2)-irreducible non-hyperbolic graphs, finding only 16 knots
and two links. The same phenomenon happens for $c=3,4$, where we
have shown that only knots and links are (0,1,2)-irreducible and
non-hyperbolic. However we refrained from classifying them
completely, confining ourselves to those in $S^3$ with $c=3$.
Since our primary interest was in hyperbolic graphs, we
decided to rule out knot components in complexity $5$, but quite
interestingly we have found some non-hyperbolic examples in this
case. Our results are summarized by Table~\ref{non:num:tab} and
the next statement:

\begin{prop}\label{non:sum:prop}
The only $(0,1,2)$-irreducible non-hyperbolic graphs $(M,G)$ with
$c(M,G)\leqslant 5$ such that $G$
has no knot component are the trivial $\theta$-graph in $S^3$, which has complexity $0$,
and five pairs in complexity $5$, where $G$ is a $\theta$-graph and $M\setminus G$
contains an embedded Klein bottle.
\end{prop}

\begin{table}[h]
\begin{center}
\begin{tabular}{l||r|r|r|r|r|r}
type & $c=0$ & $c=1$ & $c=2$ & $c=3$ & $c=4$ & $c=5$  \\ \hline\hline
knot (in $S^3$) & 3 (1) & 4 (1) & 12 (1) & -- (4) & -- (--) & -- (--) \\ \hline
2-link (in $S^3$) & 0 (0) & 1 (1) & 1 (0) & -- (1) & -- (--) & -- (--) \\ \hline
2t (in $S^3$) & 1 (1) & 0 (0) & 0 (0) & 0 (0) & 0 (0) & 5 (0) \\
\end{tabular}
\end{center}
\mycap{\textbf{Numbers of (0,1,2)-irreducible non-hyperbolic graphs.}
When $c=5$ we have not investigated graphs having knot components;
-- indicates that graphs of this type were not classified.
The graph types not mentioned were all
investigated and found to have no representative.
\label{non:num:tab}}
\end{table}

A precise description of the knots, links and graphs appearing in
Table~\ref{non:num:tab} will be provided in
Section~\ref{non:section}.

\subsection*{Some open problems}
We conclude this introduction by suggesting a few problems for further investigation.
\begin{enumerate}
\item Enumerate the first few hyperbolic graphs with parabolic
meridians in order of increasing hyperbolic volume.
\item Enumerate the first few hyperbolic $3$-manifolds of finite volume
with (compact or non-compact) geodesic boundary in order of
increasing hyperbolic volume.
\item Enumerate the first few closed hyperbolic 3-orbifolds in order of increasing complexity
as defined in~\cite{orb:compl}.
\item Enumerate the first few closed hyperbolic 3-orbifolds in order of increasing hyperbolic volume.
\item Determine the exact complexity of infinite families of knotted graphs, for example
the torus knots in lens spaces (see Conjecture~\ref{torus:conj} below).
\end{enumerate}

Note that Kojima and Miyamoto~\cite{KoMi,Mi} have already identified the
lowest volume hyperbolic $3$-manifolds with compact and non-compact
geodesic boundary. Perhaps the ``Mom technology'' introduced by Gabai, Meyerhoff 
and Milley \cite{MOM1, MOM2} may offer an approach to (1) and (2). Recent work of Martin with Gehring and Marshall \cite{GM, MM} has identified the lowest volume orientable hyperbolic 3-orbifold. 

\section{Hyperbolic geometry}\label{hyp:section}

In this section we review the main geometric notions and results
we will need in the rest of the paper.

\subsection{Hyperbolic structures with parabolic meridians} \label{parabolic:subsection}

To help classify knotted graphs, we will study hyperbolic structures
analogous to the compete hyperbolic structure on the complement of a
knot or link. Given a graph $G$ in a closed orientable $3$-manifold $M$, let $N$
be the manifold obtained from $M\setminus G$ by removing an open
regular neighbourhood of the vertex set of $G$. Thus $N$ is a
non-compact $3$-manifold with boundary consisting of 3-punctured
spheres, one corresponding to each vertex of $G$. Then we say that
$(M,G)$ has a {\em hyperbolic structure with parabolic meridians} if
$N$ admits a complete hyperbolic metric of finite volume with
geodesic boundary (with toric and annular cusps). Equivalently, the
double $D(N)$ of $N$ admits a complete hyperbolic metric of finite
volume (with toric cusps). Such a hyperbolic structure on $N$ is
unique by a standard argument using Mostow-Prasad
rigidity~\cite{bible} and Tollefson's classification~\cite{tollef} of involutions with
2-dimensional fixed point set (see~\cite{Th2} and also~\cite{FP}).

\begin{example}\label{handcuff:example}
\emph{The simplest hyperbolic handcuff graph $(S^3,G)$ can be
obtained from one tetrahedron with the two front faces folded
together and the two back faces folded together giving a
triangulation of $S^3$ with the graph $G$ contained in the
1-skeleton as shown in Figure~\ref{handcuffs}.}
   \begin{figure}[h]
    \begin{center}
    \includegraphics[scale=0.65]{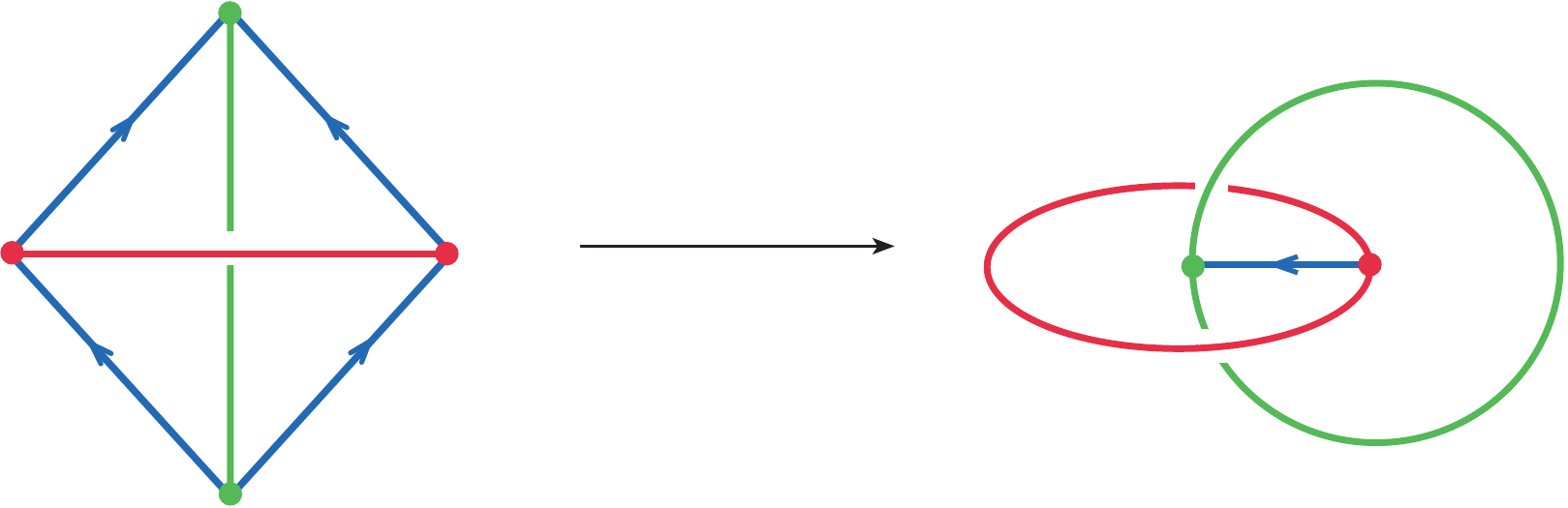}
    \mycap{\label{handcuffs}The simplest hyperbolic handcuff graph.}
    \end{center}
    \end{figure}

\emph{If we truncate the vertices of the tetrahedron until all edge
lengths are zero, the result can be realized geometrically by a
regular ideal octahedron in hyperbolic space, as shown in
Figure~\ref{oct}. We can then glue the 4
unshaded faces together in pairs so that the other 4 shaded faces form
two totally geodesic 3-punctured spheres.} 
   \begin{figure}[h]
    \begin{center}
     \includegraphics[scale=0.60]{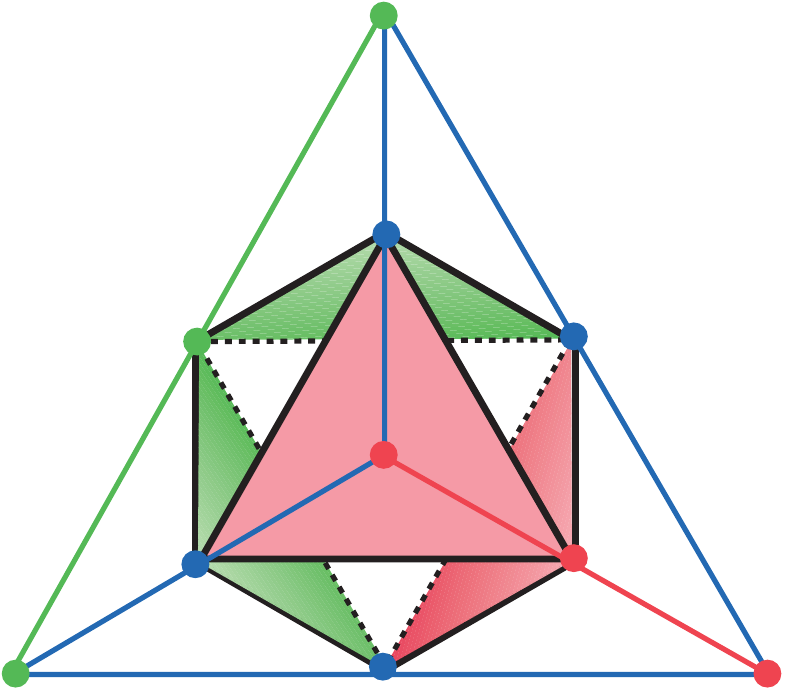}
    \mycap{\label{oct}
    Truncating the vertices of a tetrahedron produces a regular ideal octahedron
    whose unshaded faces can be glued in pairs to give a hyperbolic structure with parabolic 
    meridians on the graph of Figure~\ref{handcuffs}.}
    \end{center}
    \end{figure}

\emph{This gives a hyperbolic structure with parabolic
meridians for $(S^3,G)$ with hyperbolic volume 3.663862377... 
The work of Miyamoto and Kojima~\cite{KoMi,Mi} shows that this
is the {\em smallest} volume for trivalent graphs.
Their work also implies that a trivalent graph having this volume is obtained by
identifying the unshaded faces of an ideal octahedron as above, and
hence has complexity $1$. Therefore the handcuff graph in
Figure~\ref{handcuffs} is the unique graph of minimal volume.}
\end{example}

We next describe topological conditions for the existence of a
hyperbolic structure with parabolic meridians. Let $X$ denote
the graph exterior, \emph{i.e.}, the compact manifold obtained from $M$ by
removing an open regular neighbourhood of the graph $G$.  Then
$\bd X$ is a disjoint union of pairs of pants (corresponding to the
vertices of $G$) and  a collection of annuli and tori $P \subset
X$ (corresponding to the edges and knots in $G$). Thurston's
hyperbolization theorem for pared $3$-manifolds
\cite{Mor1, Kap} implies the following:
\begin{thm}\label{hyp_par_merid}
$(M,G)$ admits a hyperbolic structure with parabolic meridians if and only if
\begin{itemize}
\item $X$ is irreducible and homotopically atoroidal,
\item $P$ consists of incompressible annuli and tori,
\item there is no essential annulus $(A,\bd A) \subset (X, P)$, and
\item $(X,P)$ is not a product $(S,\bd S) \times [0,1]$ where $S$ is a pair of pants.
\end{itemize}
This hyperbolic structure is unique up to isometry.
\end{thm}

\begin{rem}
\emph{To obtain a hyperbolic structure with geodesic boundary on a
general pared manifold $(X,P)$,  we would need to add the
requirements that $\bd X \setminus P$ incompressible and $(X,P)$
is acylindrical  (\emph{i.e.}, every annulus $(A,\bd A) \subset (X,
\bd X \setminus  P)$ is homotopic into $\bd X$). But these
conditions follow here since $\bd X \setminus P$ consists of $3$-punctured
spheres (see~\cite[pp.~243-244]{BLP}).}
\end{rem}

The conditions for hyperbolicity simplify considerably
when $(M,G)$ is $(0,1,2)$-irreducible, as defined in the introduction.
To elucidate the notion, we say
that $(M,G)$ is:
\begin{itemize}
\item 0-\emph{irreducible} if every 2-sphere in $M$ disjoint from $G$
bounds a 3-ball in $M$ disjoint from $G$;
\item 1-\emph{irreducible} if
there exists no 2-sphere in $M$ meeting $G$
transversely in a single point;
\item 2-\emph{irreducible} if every 2-sphere in $M$ meeting $G$
transversely in two points bounds a ball in $M$ that intersects $G$ in
a single unknotted arc.
\end{itemize}
Then a graph is $(0,1,2)$-irreducible if it is $i$-irreducible
for $i=0,1,2$.

\begin{thm}\label{hyp_par_irred}
$(M,G)$ admits a hyperbolic structure with parabolic meridians if and only if
\begin{itemize}
\item $(M,G)$ is $(0,1,2)$-irreducible,
\item $X$ is homotopically atoroidal and is not a solid torus or
the product of a torus with an interval, and
\item $(M,G)$ is not the trivial $\theta$-graph in $S^3$.
\end{itemize}
\end{thm}

\begin{proof}
It is easy to check that the conditions listed are necessary for hyperbolicity.
To show that they are sufficient,
first note that $0$-irreducibility of $(M,G)$ implies that $X$ is irreducible, and
$1$-irreducibility implies that $P$ is incompressible or $X$ is a solid torus,
but the latter possibility is excluded. Moreover $(X,P)$ is not
a product $(S,\partial S) \times [0,1]$ where $S$ is a pair of pants,
because $(M,G)$ is not the trivial $\theta$-graph in $S^3$.
According to the previous theorem
we are only left to show that there cannot exist an essential annulus
$(A,\bd A) \subset (X, P)$. Suppose the contrary and note
that each of the two components of $\partial A$
is incident to either an annular or a toric component of $P$.
We show that the existence of such an annulus $A$ is impossible
by considering the three possibilities:
\begin{enumerate}
\item If $A$ is only incident to annuli of $P$, we readily see that
$2$-irreduciblity is violated.
\item If $A$ is incident to an annular
component $A' $ of $P$ and a torus component $T$ of $P$, then the
boundary of a regular neighbourhood of $A\cup T$ is another
annulus incident to $A'$ only. Again we see that
$2$-irreducibility is violated,
since the resulting sphere does not bound a
ball containing a single unknotted arc.
\item If $A$ is incident to toric components only, proceeding as in the
previous case we find one or two tori, depending on whether the toric
components are distinct or not. Homotopic atoroidality implies that these
tori must be compressible or boundary parallel in $X$. Using
irreducibility of $X$ and incompressibility of $A$, we find that $X$ is
Seifert fibred with the core circle of $A$ as a fibre and base space
either a pair of pants,
an annulus with at most one singular point, or a disc with at most two
singular points. By homotopic atoroidality, we deduce that $X$ is the
product of a torus and an interval or a solid torus, contrary to our assumptions.
\end{enumerate}
\end{proof}

\begin{cor}
If $G$ is a trivalent graph containing at least one vertex, then
$(M,G)$ is hyperbolic with parabolic meridians if and only if
$(M,G)$ is $(0,1,2)$-irreducible, geometrically atoroidal and not
the trivial $\theta$-graph in $S^3$.
\end{cor}

\subsection{Hyperbolic structures with geodesic boundary}

Let $(M,G)$ be a graph, and let $X$ denote the graph exterior as
above. Let us define $Y$ as the manifold obtained by mirroring $X$
in its non-toric boundary components, so $Y$ is either closed or
bounded by tori. Then $X$ minus its toric boundary components has a
hyperbolic structure with totally geodesic boundary if and only if
the interior of $Y$ has a complete hyperbolic structure. By
Thurston's hyperbolization theorem~\cite{Mor1,Kap}
and Mostow-Prasad rigidity (see~\cite[p.~14]{Th2}) we then have:

\begin{thm}\label{hyp_geod_bd}
$X$ minus its toric boundary components
admits a hyperbolic structure with totally geodesic boundary if
and only if  $X$ is irreducible, boundary incompressible,
homotopically atoroidal, and acylindrical. This hyperbolic structure
is unique up to isometry.
\end{thm}

Comparing
Theorems~\ref{hyp_par_merid} and~\ref{hyp_geod_bd} one easily sees
that if $X$ minus its toric boundary components
admits a hyperbolic structure with geodesic
boundary, then $(M,G)$ admits a hyperbolic structure with
parabolic meridians. The converse is however false, as most of the
pairs $(M,G)$ described below show.

\subsection{Hyperbolic orbifolds}
One of the initial motivations of our work was the study of hyperbolic
$3$-orbifolds, but the analysis of graphs turned out to be interesting enough
by itself, so we decided to leave orbifolds for the future.
However we mention them briefly here.

Given a trivalent graph $G$ in a closed $3$-manifold $M$, we obtain an orbifold $Q$
associated to $(M,G)$ by attaching an integer label $n_e \geqslant 2$ to
each edge or circle $e$ of $G$. Note that we do not impose any restrictions
on the labels $(p,q,r)$ of the edges incident to a vertex $v$, so from a topological
viewpoint $v$ gives rise either to an interior point of $Q$
(if $\frac{1}{p}+\frac{1}{q}+\frac{1}{r}>1$)
or to a boundary component of $Q$ --- a 2-orbifold of type $S^2(p,q,r)$.

We will say that $Q$ is \emph{hyperbolic} if
$M \setminus G$ admits an incomplete hyperbolic metric whose completion
has  a cone angle $\frac{2\pi}{n_e}$ along each
edge or circle $e$ in $G$.
Depending on whether $\frac{1}{p}+\frac{1}{q}+\frac{1}{r} - 1$ is positive, zero or negative,
a vertex with incoming labels $(p,q,r)$ gives rise to an interior point of $Q$ to which
the singular metric extends, to a cusp of $Q$, or to a totally geodesic boundary component of $Q$.

The main connections between orbifold hyperbolic structures and those we
deal with in this paper are as follows:

\begin{itemize}
\item  If $(M,G)$ has a hyperbolic orbifold structure for some choice of
labels $n_e$, then $(M,G)$ admits a hyperbolic structure with
parabolic meridians.

\item If $(M,G)$ admits a hyperbolic structure with parabolic meridians
then the corresponding orbifolds are hyperbolic provided all
labels are sufficiently large; moreover the structure
with parabolic meridians
can be regarded as the limit of the orbifold hyperbolic structures as all labels
tend to infinity.
\end{itemize}

The first assertion follows from Theorem~\ref{hyp_par_merid} by
topological arguments only (see~\cite[Prop.~6.1]{BLP}), while the second
one is a consequence of Thurston's hyperbolic Dehn surgery theorem
(see~\cite{BLP, CHK} for details).

\subsection{Algorithmic search for hyperbolic structures}
As already mentioned, the hyperbolic structures and related
invariants on the 123 pairs of our census have been obtained using
the computer program \emph{Orb}~\cite{orb}. More details on this
program will be provided below, but we outline here the underlying
theoretical idea (due to Thurston~\cite{bible}) of the algorithmic
construction of a hyperbolic structure with geodesic boundary on a
pared manifold $(X,P)$, where $X$ is compact but not closed and $P$
is a collection of tori and annuli on $\partial X$.

The starting point is a (suitably defined) \emph{ideal triangulation} of $(X,P)$, namely
a realization of $(X,P)$ as a gluing of \emph{generalized ideal tetrahedra}. Each of these
is a tetrahedron with its vertices removed and, depending
on its position with respect to $\partial X$ and $P$, perhaps entire
edges and/or open regular neighbourhoods of vertices also removed.
The next step is to choose a realization of each of these tetrahedra as
a geodesic generalized ideal tetrahedron in hyperbolic $3$-space. These realizations
are parameterized by certain moduli, and the condition that the
hyperbolic structures on the individual tetrahedra match up to give a hyperbolic
structure on $(X,P)$ translates into equations in the moduli. The algorithm
then consists of changing the initial moduli using Newton's method until the (unique) solution
of the equations is found.

When $M$ is closed one can search for its hyperbolic structure using a similar method,
starting from a decomposition of $M$ into compact tetrahedra~\cite{Cas}.

\subsection{Canonical cell decompositions}
Whenever a hyperbolic manifold $X$ is not closed, it admits a canonical
decomposition into geodesic hyperbolic polyhedra, which allows one to very efficiently
compute its symmetry group and compare it for equality with another such manifold.
The decomposition was defined by Epstein and Penner~\cite{EpPe} when $\partial X=\emptyset$
but $X$ has cusps, and by Kojima~\cite{Koj1,Koj2} when $\partial X\neq\emptyset$.
We will now briefly outline the latter construction.

Begin with the
geodesic boundary components of $X$ and very small horospherical
cross sections of any torus cusps of $X$, and expand these
surfaces at the same rate until they bump to give a 2-complex (the
cut locus of the initial boundary surfaces). Then dual to this
complex is the {\em Kojima canonical decomposition} of $X$ into
generalized ideal hyperbolic polyhedra. This is independent of the
choice of horosphere  cross sections provided they are chosen
sufficiently small, and gives a complete topological invariant of
the manifold.

Thus two finite volume hyperbolic $3$-manifolds with geodesic
boundary are isometric (or, equivalently, homeomorphic)
if and only if their Kojima canonical
decompositions are combinatorially the same; and the symmetry
group of isometries of such a manifold is the group of
combinatorial automorphisms of the canonical decomposition.
Similarly, two graphs admitting hyperbolic structures with parabolic meridians are
equivalent if and only if there is a combinatorial isomorphism
between their canonical decompositions taking meridians to
meridians; and the group of symmetries of such a graph is the group
of combinatorial automorphisms of the canonical decomposition
taking meridians to meridians.

\subsection{Arithmetic invariants}\label{arith_invar}
Let us first note that a hyperbolic structure on an
orientable $3$-manifold without boundary corresponds to
a realization of the manifold as the quotient of hyperbolic
space $\H^3$ under the action of a discrete group $\Gamma$
of orientation-preserving isometries of
$\H^3$. If the manifold has boundary, $\H^3$ should be replaced
by a $\Gamma$-invariant intersection of closed half-spaces in $\H^3$.
Moreover for any given hyperbolic $3$-manifold, the group
$\Gamma$ is well-defined up to conjugation within
the full group of orientation-preserving isometries of
$\H^3$, which is isomorphic to $\mbox{PSL}(2,\C)$.

If  $\Gamma$ is a discrete subgroup of $\mbox{PSL}(2,\C)$, then
the {\em invariant trace field} $k(\Gamma)\subset \C$ is the field
generated by the traces of the elements of $\Gamma^{(2)} =
\{\gamma^2\mid \gamma\in\Gamma\}$ lifted to $\mbox{SL}(2,\C)$.
This is a commensurability invariant of $\Gamma$ (unchanged if
$\Gamma$ is replaced by a finite index subgroup). Further, if
$\H^3/\Gamma$ has finite volume then it follows from Mostow-Prasad
rigidity that $k(\Gamma)$ is a number field, \emph{i.e.}, a finite degree
extension of the rational numbers $\Q$.  (See~\cite{MacReid} for
an excellent discussion and proofs.)

If a trivalent graph $(M,G)$ admits a hyperbolic structure $N$ with parabolic meridians, then $N$ is the convex hull of $\H^3/\Gamma$ where $\Gamma$ is a discrete subgroup of $\mbox{PSL}(2,\C)$. Thus $k(\Gamma)$ is an invariant of $(M,G)$. Now the double $D(N)$ (defined at the start of Subsection \ref{parabolic:subsection}) has the form $\H^3/\Gamma_1$, where $\Gamma_1$ is a Kleinian group containing $\Gamma$. 
Since $D(N)$ is hyperbolic with finite volume, $k(\Gamma_1)$ is an algebraic number field. Hence the subfield $k(\Gamma)$ is also an algebraic number field. We compute this by combining \emph{Orb} with a
modified version of Oliver Goodman's program \emph{Snap} (\cite{snap}).

\emph{Snap} begins with generators and relations for $\Gamma$, and a
numerical approximation to $\Gamma$ provided by \emph{Orb}.
It first
refines this using Newton's method to obtain a high precision
numerical approximation to $\Gamma$, and
then tries to find exact
descriptions of matrix entries and their traces as algebraic
numbers using the LLL-algorithm. Finally \emph{Snap} verifies that we
have an exact representation of $\Gamma$ by checking that the
relations for $\Gamma$ are satisfied using exact calculations in a
number field, and computes the  invariant trace field $k(\Gamma)$ and
associated algebraic invariants. (See~\cite{CGHN} for a detailed
description of \emph{Snap}.)

\section{Complexity theory}\label{compl:section}

A theory of complexity for 3-orbifolds, mimicking Matveev's theory
for manifolds~\cite{Matv:AAM}, was developed in~\cite{orb:compl}.
Removing all references to edge orders and their contributions to
the complexity, one deduces a theory of complexity for 3-valent
graphs embedded in closed orientable $3$-manifolds. In this
paragraph we will summarize the main features of this theory.
The main ideas of this theory are as follows:

\begin{itemize}

\item Triangulations are the best way to manipulate 3-dimensional
topological objects by computer.

\item Therefore, the minimal number of tetrahedra required to triangulate
an object gives a very natural measure of the complexity of the object.

\item However, there exists another definition of complexity, based
on the notion of simple spine. A triangulation, via a certain ``duality,'' gives
rise to a simple spine, therefore complexity defined via spines is not
greater than complexity defined via triangulations.

\item Simple spines are more flexible than triangulations. In particular,
there are more general non-minimality criteria for simple
spines than for triangulations. More specifically, there are instances
where a triangulation may appear to be minimal (as a triangulation)
whereas the dual spine is obviously not minimal (as a simple spine).

\item A theorem ensures that for a hyperbolic object a minimal simple spine
is always dual to a triangulation.

\item As a conclusion, if one wants to carry out a census of hyperbolic
objects in order of increasing complexity, one deals by computer
with triangulations, but one discards triangulations to which, via duality,
the stronger non-minimality criteria for spines apply. This is because,
thanks to the theorem, such a triangulation encodes either a non-hyperbolic
object or a hyperbolic object that has been met earlier in the census.

\end{itemize}
We will now turn to a more
detailed discussion.

\subsection{Simple spines and complexity}
To proceed with the key notions and results we recall
a definition given in the Introduction. We call
\emph{simple}\footnote{In~\cite{Matv:AAM} such a polyhedron was
originally called \emph{almost simple}, while the term
\emph{simple} was employed for almost special polyhedra, see
Subsection~\ref{special:subsection}.} a compact polyhedron $P$ (in
the PL sense~\cite{RS}) such that the link of each point is a
subset of the 1-skeleton of the tetrahedron. We denote by $V(P)$
the set of points of $P$ having the whole 1-skeleton of the
tetrahedron as a link, and we note that $V(P)$ is a finite set.

\begin{defn} \label{simple_spine_def}
A \emph{simple spine} of a trivalent graph $(M,G)$  is
a simple polyhedron $P$ embedded in $M$ in such a way that:
\begin{enumerate}
\item
$G$ intersects $P$ transversely.
(In particular, $P\cap G$ consists of a finite number of
points that are not vertices of $G$).
\item Removing an open regular neighbourhood of $P$ from $(M,G)$ gives a finite collection
of balls, each of which intersects $G$ in either
\begin{itemize}
\item the empty set, or \item a single unknotted arc of $G$, or
\item a vertex of $G$ with unknotted strands leaving the vertex
and reaching the boundary of the ball. (See
Figure~\ref{balls:fig}.)
\end{itemize}
\end{enumerate}
\end{defn}

It is very easy to see (and it will follow from the duality with
triangulations in Proposition~\ref{duality:mfld:prop}) that each
$(M,G)$ admits simple spines. Therefore the \emph{complexity} of
$(M,G)$, that we define as
$$c(M,G)=\min\big\{\#V(P):\ P\ \textrm{simple\ spine\ of}\ (M,G)\big\},$$
is a finite number.

\subsection{Special spines and duality} \label{special:subsection}
To illustrate the relation between spines and triangulations, we need to
introduce two subsequent refinements of the notion of simple polyhedron.
We will say that $P$
is \emph{almost-special} if it is a compact polyhedron and each of its points
has one of the following sets as a link:
\begin{enumerate}
\item The $1$-skeleton of the tetrahedron with two open opposite edges removed
(a circle);
\item The $1$-skeleton of the tetrahedron with one open edge removed
(a circle with a diameter);
\item The $1$-skeleton of the tetrahedron
(a circle with three radii).
\end{enumerate}
The corresponding local structure
of an almost-special polyhedron is
shown in Figure~\ref{almost:special:fig}.

    \begin{figure}[h]
    \begin{center}
    \includegraphics[scale=1.0]{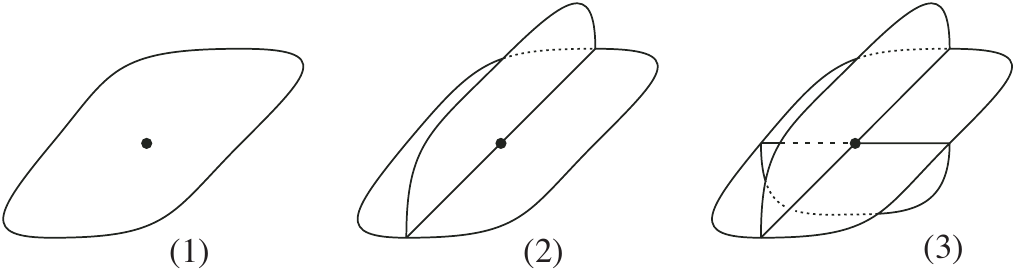}
    \mycap{\label{almost:special:fig}Local structure of an almost-special polyhedron.}
    \end{center}
    \end{figure}
    
Besides the set $V(P)$ of vertices already introduced
above for simple polyhedra, we can define for an almost-special $P$ the
\emph{singular set}, given by the non-surface points and denoted
by $S(P)$. We remark that $S(P)$ is a 4-valent graph
with vertex set $V(P)$. Note also that if
  $P$ is an almost-special spine of $(M,G)$, by the transversality
  assumption, $G$ intersects $P$ away from $S(P)$.

An almost-special polyhedron $P$ is called \emph{special} if $P\setminus S(P)$
is a union of open discs and $S(P)\setminus V(P)$ is a union of open segments.
A \emph{special spine} of a graph $(M,G)$ is a simple spine which, in addition,
is a special polyhedron.

The following result, which refers to the case of manifolds
without graphs embedded in them, has been known for a long time.
We point out that we use the term \emph{triangulation} for a (closed,
connected, orientable) $3$-manifold $M$
in a generalized (not
strictly PL~\cite{RS}) sense. Namely, we mean a realization of $M$
as a simplicial pairing  between the faces of a finite union of
tetrahedra, \emph{i.e.}, we allow multiple and self-adjacencies
between tetrahedra.

\begin{prop}\label{duality:mfld:prop}
Given a $3$-manifold $M$, for each triangulation $\calT$ of $M$
define $\Phi(\calT)$ as the $2$-skeleton of the cell decomposition
dual to $\calT$, see Figure~\ref{duality:fig}.
    \begin{figure}
    \begin{center}
    \includegraphics[scale=1.0]{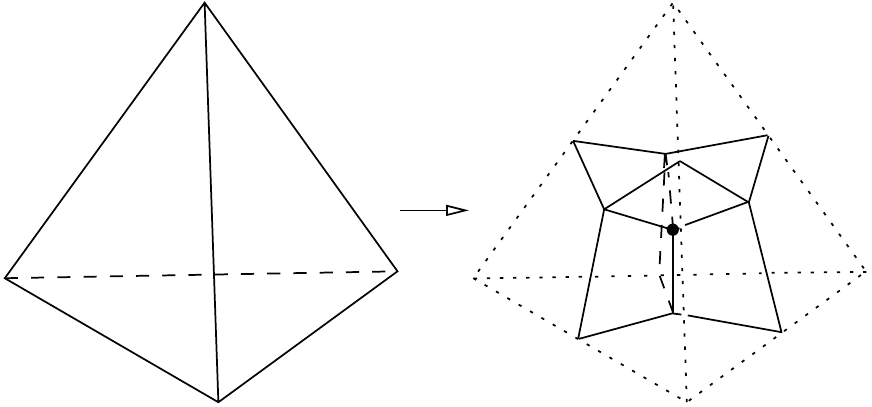}
    \mycap{\label{duality:fig}Duality between triangulations and special spines.}
    \end{center}
    \end{figure}
Then $\Phi$ defines a
bijection between the set of (isotopy classes of) triangulations
of $M$ and the set of (isotopy classes of) special spines of $M$.
\end{prop}

\subsection{(Efficient) triangulations of graphs}
We now turn to graphs $(M,G)$, and we define a \emph{triangulation}
of $(M,G)$ to be a (generalized)
triangulation $\calT$ of $M$ which contains $G$ as a subset of
its 1-skeleton. We will further say that $\calT$ is \emph{efficient} if
it has precisely one vertex at each vertex of $G$, one on each knot component of
$G$, and no other vertices.

The following easy result shows that under suitable conditions
Proposition~\ref{duality:mfld:prop} has a refinement to graphs:

\begin{prop}\label{duality:graph:prop}
For a simple spine $P$ of a graph $(M,G)$ the following conditions are equivalent:
\begin{itemize}
\item $P$ is dual to a triangulation of $(M,G)$;
\item $P$ is special, $G$ intersects $P$ transversely away from $S(P)$, and each component
of $P\setminus S(P)$
intersects $G$ at most once.
\end{itemize}
\end{prop}

\subsection{Minimal spines}
A simple spine $P$ of a graph $(M,G)$ is called \emph{minimal}
if it has $c(M,G)$ vertices and no subset of $P$ is also a spine
of $(M,G)$. The success of the strategy based on complexity theory
(as outlined at the beginning of this section) for the enumeration of hyperbolic
graphs depends on the next three results. They require the
concept of $(0,1,2)$-irreducibility
defined in the introduction. The first one is part of
Theorem~\ref{hyp_par_irred},
the next two easily follow from~\cite[Theorem 2.6]{orb:compl}.

\begin{prop}
If $(M,G)$ is hyperbolic with parabolic meridians then $(M,G)$
is $(0,1,2)$-irreducible.
\end{prop}

\begin{prop}\label{c=0:irred:prop}
The $(0,1,2)$-irreducible graphs $(M,G)$ with $c(M,G)=0$ are those described as follows
and illustrated in Figure~\ref{c=0:fig}:
\begin{itemize}
\item $M$ is either $S^3$, or $L(3,1)$, or $\matP^3$, and $G$ is either empty
or the core of a Heegaard solid torus of $M$;
\item $M$ is $S^3$ and $G$ is the trivially embedded $\theta$-graph.
\end{itemize}
\end{prop}

    \begin{figure}
    \begin{center}
    \includegraphics[scale=1.0]{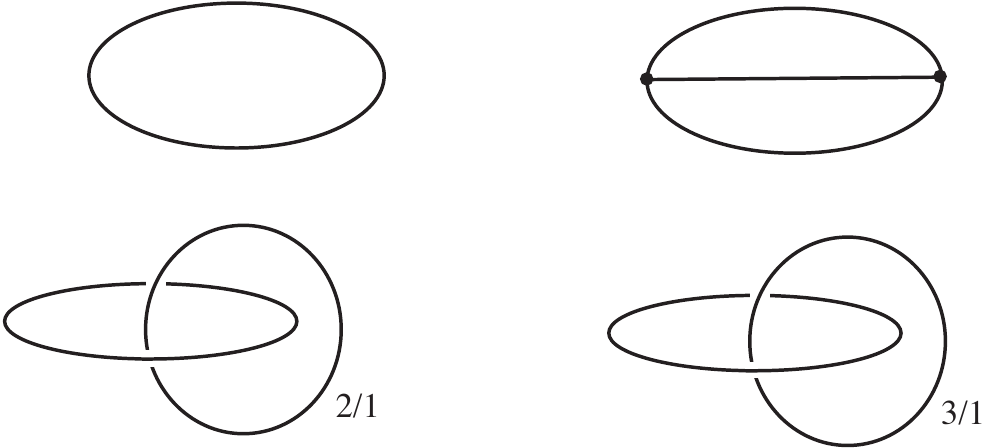}
    \mycap{\label{c=0:fig}
    \textbf{The $(0,1,2)$-irreducible graphs of complexity 0.} Here and below a knot component carrying
    a fractional label should be understood as a surgery instruction~\cite{Rolfsen}.
    In particular, it is not actually part of the graph.}
    \end{center}
    \end{figure}

\begin{thm}\label{good:min:spin:thm}
Let $(M,G)$ be a graph with $c(M,G)>0$. Then the following are equivalent:
\begin{itemize}
\item $(M,G)$ is $(0,1,2)$-irreducible;
\item $(M,G)$ admits a special minimal spine;
\item Every minimal spine of $(M,G)$ is special and dual to it
there is an efficient triangulation of $(M,G)$.
\end{itemize}
\end{thm}

\subsection{Non-minimality criteria}
The following result was used for the enumeration of candidate triangulations
of $(0,1,2)$-irreducible graphs, as explained in more detail in the next section:

\begin{prop}\label{edge:cond:prop}
Let $\calT$ be a triangulation of a graph $(M,G)$, and let $P$ be the
special spine dual to $\calT$. Suppose that in $\calT$ there is
an edge not lying in $G$ and
incident to $i$ distinct tetrahedra, with $i\leqslant 3$.
Then $P$ is not minimal.
\end{prop}

\begin{proof}
We will show that we can perform a move on $P$ leading to
a simple spine of $(M,G)$ with fewer vertices than $P$.

For $i=3$ we do not even need to use spines, the move exists already
at the level of triangulations: it is the famous Matveev-Piergallini $3\to 2$
move~\cite{Matv-move,Pierg}  illustrated in Figure~\ref{3to2:fig}.
We only need to note that after the move we still have a triangulation
of $(M,G)$ because the edge that disappears with the move does
not lie in $G$.

    \begin{figure}[h]
    \includegraphics[scale=0.8]{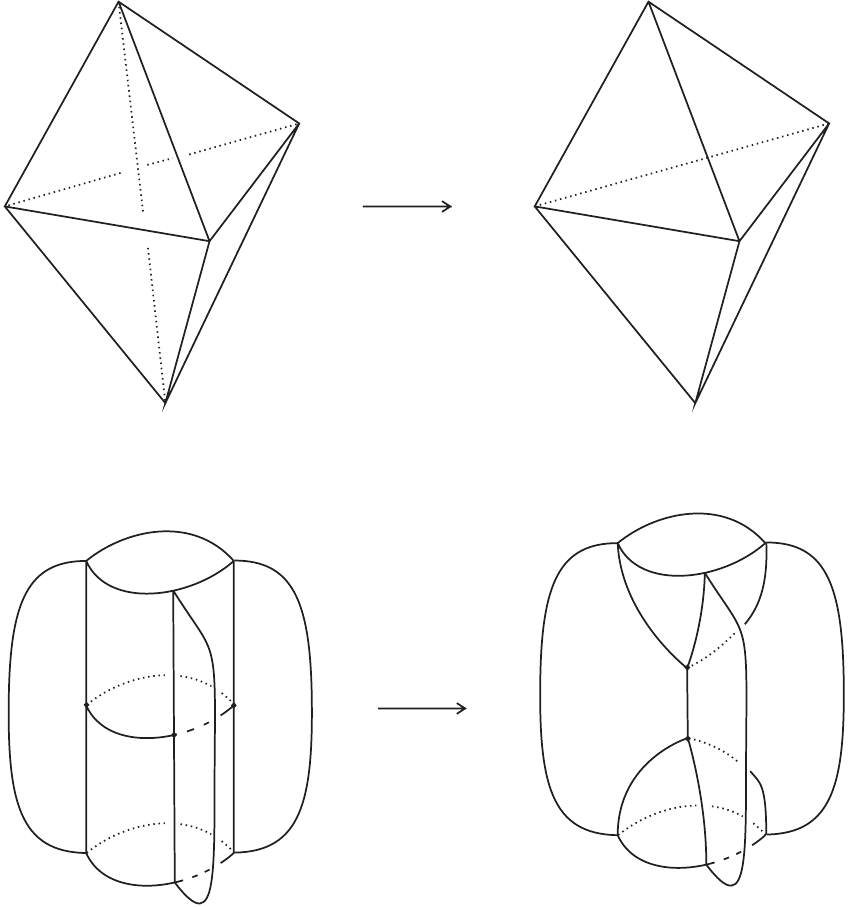}
    \begin{center}
    \mycap{\label{3to2:fig} The $3\to 2$ move on triangulations and its dual version for spines.}
    \end{center}
    \end{figure}

For $i=1,2$ we do need to use spines. The moves we apply (a $1\to
0$ and a $2\to 0$ move) are illustrated in
Figure~\ref{2to0:1to0:fig}.
Both moves involve the removal
of the component $R$ of $P\setminus S(P)$
dual to the edge of the statement, and the result of the move is still a spine
of $(M,G)$ because $G$ does not meet $R$. We note that the $2\to 0$ move
leads to an almost-special polyhedron, but it can create
a spine with
an annular non-singular component, in which case the spine is not dual to
a triangulation. The $1\to 0$ move gives a spine which is not
almost-special.
\end{proof}
    \begin{figure}[h]
    \includegraphics[scale=1.0]{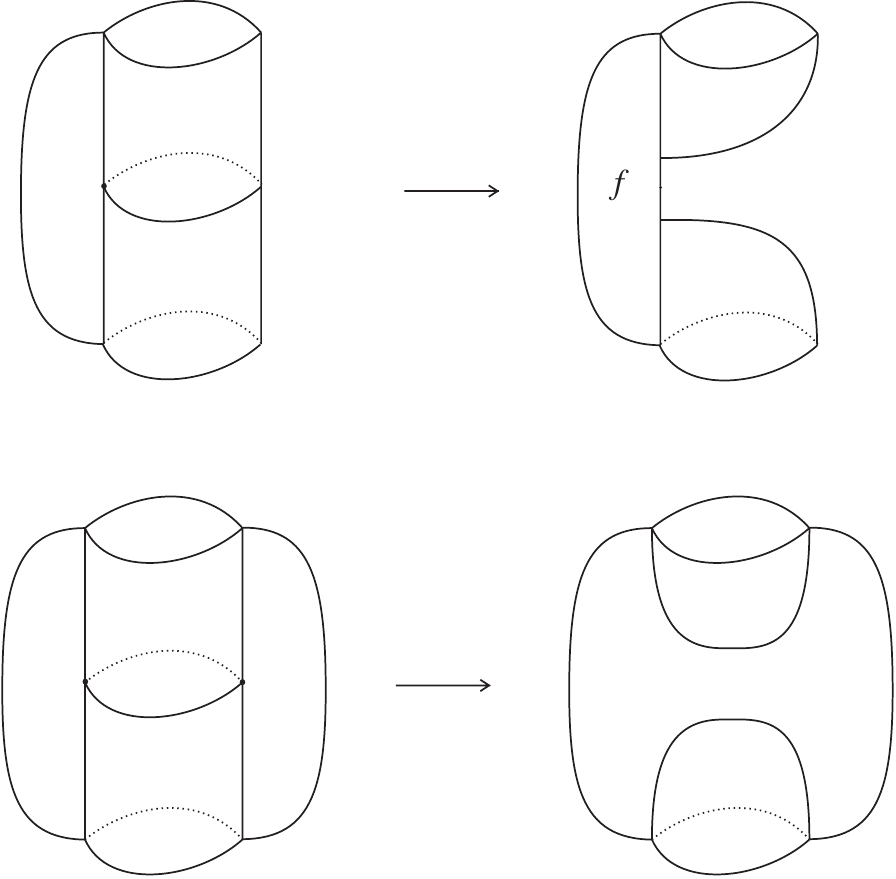}
    \begin{center}
\mycap{\label{2to0:1to0:fig}
\textbf{The $1\to 0$ and the $2\to 0$ moves on spines.}
Both these moves transform a special spine $P$ into a simple spine which is not
necessarily special. If $P$ has at least $2$ vertices, both moves
destroy at least 2 vertices of $P$: the $2\to 0$ move destroys
precisely two; the $1\to 0$ move can be completed by collapsing the
face $f$, which is necessarily adjacent to at least another vertex
of $P$ that disappears after the collapse.}
    \end{center}
    \end{figure}

\begin{rem}
{\em Sometimes the non-minimality criteria of the previous
proposition do not apply directly, but only after a modification
of the triangulation. For instance, a triangulation $T$ with $n$
tetrahedra may be transformed into one $T'$ with $n+1$ tetrahedra
via a $2\to 3$
move: if $T'$ contains an edge incident to $1$ or $2$
distinct tetrahedra, the dual spine $P'$ can be transformed into a
simple spine with at most $n-1$ vertices by applying one of the
moves in Figure~\ref{2to0:1to0:fig}. Therefore the original
triangulation $T$ is not minimal.}
\end{rem}

\subsection{Complexity of the complement}
Matveev's complexity~\cite{Matv:AAM} is defined for every compact
$3$-manifold, with or without boundary. The complement $X$
of an open regular neighbourhood of a graph $G$ in a
closed $3$-manifold $M$ therefore has
a complexity, which is related to $c(M,G)$ as follows.

\begin{prop} \label{complement:prop}
For any graph $(M,G)$ we have
$$c\big(X)\leqslant c(M,G).$$
If $(M,G)$ is $(0,1,2)$-irreducible with $c(M,G) \ne 0$ and $G\neq \emptyset$, then
$$c(X)<c(M,G).$$
\end{prop}

\begin{proof}
If $P$ is a minimal simple spine of $(M,G)$ then
the graph $G$ intersects $P$ in a finite number of points.
Removing from $P$ open regular
neighbourhoods of these points gives a simple
polyhedron $P'\subset P$ which is a spine of $X$
with the same vertices as $P$. Therefore $c\big(X)\leqslant c(M,G)$.

If $(M,G)$ is $(0,1,2)$-irreducible, $G\ne \emptyset$ and $c(M,G) \ne 0$,
then Theorem~\ref{good:min:spin:thm} shows that a minimal simple
spine $P$ of $(M,G)$ is special and $G \cap P$ consists of some $k \geqslant 1$ points
belonging to the interior of $k$ distinct disc components of $P\setminus S(P)$.
Removing these $k$ discs we get a simple spine of $X$ with
strictly fewer vertices than $P$.
\end{proof}

\begin{rem}\label{complement:rem}
{\em A compact $3$-manifold which admits a complete hyperbolic
metric with geodesic boundary and finite volume (after removing
the tori from its boundary) has complexity at least $2$,
see~\cite{Matv:AAM, CaHiWe, FriMaPe}. This explains why the first
hyperbolic knots $(M,G)$ have $c(M,G)\geqslant 3$ (see
Tables~\ref{hyp:num:tab} and~\ref{non:num:tab}). Analogously, the
first graphs $(M,G)$ whose complement is hyperbolic with geodesic boundary
must have $c(M,G)\geqslant 3$.
(In fact they have
complexity $5$, see
Subsection~\ref{geodesic:subsection}.)}
\end{rem}

\section{Computer programs and obstructions to hyperbolicity}\label{compu:section}

In this section we describe the Haskell code
we have written to enumerate triangulations, and the computer program \emph{Orb}
we have used to investigate hyperbolic structures. We also describe how non-hyperbolic graphs were identified (see also Section \ref{non:section} below).

\subsection{Enumeration of marked triangulations}\label{enu:section}
Thanks to Theorem~\ref{good:min:spin:thm} and the other results
stated in the previous section, the enumeration of
$(0,1,2)$-irreducible graphs of complexity $n$ can be performed by
listing all efficient triangulations with $n$ tetrahedra satisfying
some minimality criteria. This was done via a separate program,
written in Haskell~\cite{Has}, which suitably adapts the strategy
already used in similar censuses (\emph{e.g.}~\cite{MaPe:c9, Matv:last}).

A triangulation of a graph $(M,G)$ can be encoded as a
triangulation of $M$ with some marked edges. A triangulation here
is just a gluing of tetrahedra, which can be described via a connected
$4$-valent graph (the incidence graph of the gluing) having a
label on each edge encoding how the corresponding triangular faces
are identified (there are $3!=6$ possibilities).

A first count gives $c_n\cdot 6^{2n} = c_n\cdot 36^n$
triangulations to check, where $c_n$ is the number of $4$-valent
graphs with $n$ vertices (and $2n$ edges), shown in
Table~\ref{4_valent_graphs:table}. On each triangulation there are
$2^e = 2^{n+v}$ distinct markings of edges, where $e$ is the number of edges
and $v$ is the number of vertices in
the triangulation of $M$. Since there are at least $2$ triangles in the link of each vertex,
$v \leqslant 2n$, and $e \leqslant 3n$.
There are therefore up to $c_n\cdot
36^n \cdot 2^{3n} = c_n\cdot 288^n$ marked triangulations to check.
This number is already too big for $n=3$, so in order to simplify the
problem we used some tricks.

\begin{table}
\begin{center}
\begin{tabular}{c|ccccc}
$n$ & $1$  &  $2$ & $3$ & $4$ & $5$ \\ \hline
$c_n$ & $1$ & $2$ & $4$ & $10$ & $28$ \\
$c_n'$ & $1$ & $3$ & $5$ & $18$ & $56$
\end{tabular}
\end{center}
\mycap{\label{4_valent_graphs:table}
The number $c_n$ of $4$-valent graphs with $n$ vertices, and $c_n'$ of $4$-valent graphs with oriented vertices.}
\end{table}

We are only interested in orientable manifolds $M$. We can therefore orient each tetrahedron and
require the identifications of faces to be orientation-reversing. This reduces the number of possible
labels on edges from $6$ to $3$, and the number of
triangulations to $c_n'\cdot 3^{2n}=c_n'\cdot 9^n$, where $c_n'$ is the number of $4$-valent graphs with
``oriented'' vertices: each vertex has a fixed parity of orderings of the incident edges.
For a fixed $4$-valent graph $G$ with $n$ vertices, the vertices can be oriented
in $2^n$ different ways,
but up to the symmetries of $G$ the number of
distinct orientations
typically turns out to be very small.
This explains why $c_n'$ is actually much less than $2^n\cdot c_n$, as shown in the table.

We selected from the resulting list of triangulations only those
yielding closed manifolds. Finally, on each triangulation
we \emph{a priori} had $2^{e}$ distinct markings on edges to analyze.
Proposition~\ref{edge:cond:prop} was used to discard
many of these: in a triangulation dual to a minimal spine
an edge incident to at most $3$ distinct
tetrahedra is necessarily marked. It remained then to check which
markings give rise to efficient triangulations.

\subsection{``Orb''}
Hyperbolic structures were computed using the program \emph{Orb} written
by Damian Heard~\cite{He, orb}. This program builds on
ideas of Thurston, Weeks, Casson and others to find hyperbolic
structures and associated geometric invariants for a large class
of 3-dimensional manifolds and orbifolds. The program begins with
a triangulation of the space with the singular locus or graph contained in the
1-skeleton and tries to find shapes of generalized hyperbolic
tetrahedra (with vertices inside, on, or outside the sphere at
infinity) which fit together to give a hyperbolic structure.

The generalized hyperbolic tetrahedra are described by using one
parameter for each edge in the triangulation. For a general
tetrahedron a lift to Minkowski space is chosen, then the
parameters are Minkowski inner products of the vertex positions.
For compact tetrahedra,
each parameter is just the hyperbolic cosine of the edge length. For each ideal
vertex the lift to Minkowski space determines a horosphere centred
at the vertex; for each hyperideal vertex a geodesic plane
orthogonal to the incident faces is
determined. Then the edge parameters are simple functions of the
hyperbolic distances between these surfaces.

Given the edge parameters, all dihedral angles of the tetrahedra
are determined. Moreover the parameters give a global hyperbolic structure
if and only if the sum of the dihedral angles around each edge is $2\pi$
(or the desired cone angle, in the orbifold case).
This gives a system of equations that \emph{Orb} solves numerically using
Newton's method, starting with suitable regular generalized
tetrahedra as the initial guess.

Once a hyperbolic structure is found, \emph{Orb} can compute many
geometric invariants including volumes, the Kojima canonical
decompositions, and symmetry groups. This uses methods based on
ideas of Weeks~\cite{We1}, Ushijima~\cite{ushi:tilt} and
Frigerio-Petronio~\cite{FP}, too complicated to be reproduced here.

After computing hyperbolic structures numerically using \emph{Orb}, we
checked the correctness of the results by using Jeff Weeks'
program \emph{SnapPea}~\cite{snappea} to calculate complete hyperbolic
structures on the manifolds with torus cusps obtained by doubling
along all 3-punctured sphere boundary components.

Finally, we verified the results by using Oliver Goodman's program
\emph{Snap}~\cite{snap, CGHN} to find exact hyperbolic
structures. This provides a proof that the hyperbolic structures
are correct and allows us to compute associated arithmetic
invariants (including invariant trace fields), as already mentioned 
in Subsection \ref{arith_invar} above.

\subsection{Non-hyperbolic knots and links}
Many knots and links in the census turned out to be torus links in lens
spaces, see Subsection~\ref{non:knots:subsection} below. From
$c=3$, we then decided to rule out the non-hyperbolic knots and links from
our census (except for those in $S^3$ at $c=3$); this helped a
lot in simplifying the classification. Many non-hyperbolic knots and links were
easily identified by the following criterion:

\begin{rem} {\em If the complexity of the complement is at most $1$ then the
link is not hyperbolic by Remark~\ref{complement:rem}.
This holds for instance if there are $n$ tetrahedra and the marked
edge of the triangulation is incident to at least $n-1$ of them
(see the proof of Proposition~\ref{complement:prop}).}
\end{rem}

The remaining knots and links were shown to be non-hyperbolic by examining their
fundamental groups with the help of the following observations.

\begin{lemma} \label{group_1} Let $M$ be an orientable finite volume hyperbolic
$3$-manifold, and let $a,b,c\in \pi_1(M)$. Then
\begin{enumerate}
\item[(i)] if $[a^p,b^q]=1$ for some integers $p,q \ne 0$ then $[a,b]=1$,
\item[(ii)] if $[a,b]=1$ and $b=cac^{-1}$ then $a=b$.
\end{enumerate}
\end{lemma}
\begin{proof}
The results are clear if $a, b$ or $c$ is the identity, so we may
assume that $a,b$ and $c$ correspond to loxodromic or parabolic
isometries of $\H^3$.

In part (i), the elements $a^p,b^q$ must have the same axis or fixed
point at $\infty$. Since $p,q\ne0$  the same is true for $a$ and
$b$, so $a$ and $b$ commute.

In part (ii), $a$ and $b$ have the same fixed point set $F$ on the
sphere at infinity, and $c$ takes $F$ to itself. Since $c$ is not
elliptic, it must fix each point of $F$. Thus $c$ has the same
axis or fixed point at $\infty$ as $a$ and $b$, so it commutes with
them.\end{proof}

\begin{lemma} \label{group_2} Let $M$ be an orientable finite volume hyperbolic $3$-manifold.
Then $\pi_1(M)$ cannot have a presentation of the
form
\begin{enumerate}\item[(i)] $\langle a,b \mid a^n(a^pb^q)^k=1 \rangle$
where $k,n,p,q$ are integers with $k,n,q\ne0$, or
\item[(ii)]
$\langle a,b \mid a^2b^{-1}a^{-1}b^2a^{-1}b^{-1}=1 \rangle$.
\end{enumerate}
\end{lemma}
\begin{proof} (i) If $a^n(a^pb^q)^k=1$, then $[a^n,a^pb^q]=1$ by part (i) of
Lemma~\ref{group_1}.
Hence $[a^n,b^q]=1$ and $[a,b]=1$,
again by part (i) of
Lemma~\ref{group_1}. So the group would be abelian, which
is impossible.

(ii) The group has a presentation $$\langle a,b,x,y \mid x=ab^{-1},
y=a^{-1}b,[x,y]=1 \rangle.$$ We can rewrite this as $$\langle a,x,y
\mid x=ay^{-1}a^{-1}, [x,y]=1 \rangle.$$ Hence $x=y^{-1}$ by part (ii) of Lemma
\ref{group_1} and $[a,x]=1$. So the group would be abelian, which is
again impossible.
\end{proof}

Among the knots and links up to complexity 4 for which Orb did not
find a hyperbolic structure, all but one of the complements had
a fundamental group with presentation of the form
$\langle a, b \mid [a^n, b^m] = 1\rangle$, or $\langle a,b \mid
a^n(a^pb^q)^k=1 \rangle$. These all correspond to non-hyperbolic links by the
Lemmas above. The one remaining knot had a presentation as in
part (ii) of Lemma~\ref{group_2}, so is also non-hyperbolic.

\subsection{Non-hyperbolic graphs}
For graphs with at least one vertex, we first eliminated all triangulations whose
dual spines had non-minimal complexity hence were
either reducible or occurred earlier in our list. This left a handful of
examples for which \emph{Orb}  failed to find a hyperbolic structure.
These were first examined using Jeff Weeks' program SnapPea, by
constructing triangulations of the manifolds with torus cusps
obtained by doubling along the 3-punctured sphere boundary components.
We used SnapPea's ``splitting'' function to look for incompressible
Klein bottles and tori in the doubles. This suggested that
incompressible Klein bottles were present in the original graph
complements.
We then verified this and showed that these examples were indeed non-hyperbolic
by theoretical means, as explained below in Section~\ref{non:section}.

\section{Hyperbolic census details}\label{res:section}
In this section we will expand on the information given in
Table~\ref{hyp:num:tab}, providing details of all the 123
hyperbolic graphs up to complexity 5. Pictures of the hyperbolic graphs
up to complexity 4
will be shown in Section~\ref{fig:section}.

\subsection{Name conventions}
For future reference, we have chosen a name for each of the graphs we have found. The name has the form
$$ng\_\,c\_\,i$$
where $n$ is the number of vertices of the graph, $g$ is a string
describing the abstract graph type, $c$ is the complexity, and $i$
is an index (starting from $1$ for any given $ng\_\,c$). We have
found in our hyperbolic census only 6 graph types, described above
in Figure~\ref{g:types:fig}, so a string of one letter only (or the
empty string, for knots) was sufficient to identify them.
For graphs with 2 vertices, the letters $t$ and $h$ were
suggested by the common names ``$\theta$-graph'' and
``handcuffs''. The choice of letters was arbitrary for graphs with $4$ vertices.

\subsection{Organization of tables}
We will give separate tables for
$\theta$-graphs, handcuffs, $4$-vertex graphs, and knots.
Within each table, graphs are always arranged in increasing order
of their hyperbolic volumes.
For graphs having vertices, the columns of the
tables respectively contain:
\begin{enumerate}
\item The name of the graph $(M,G)$.
\item The volume of the hyperbolic structure with parabolic meridians on $M\setminus G$.
\item A description of the cells of the Kojima
canonical decomposition for this structure.
When all these cells are tetrahedra we simply indicate their number,
otherwise we add an asterisk in the table and provide additional information separately.
\item The symmetry group of $(M,G)$, with $D_n$ denoting the
dihedral group with $2n$ elements.
\item Whether $(M,G)$ is chiral (c) or amphichiral (a). 
\item The name of the underlying space $M$. This is almost always a lens space;
otherwise it is a Seifert fibred space which we describe in
the usual way (as in~\cite[p.~406]{Matv:book}).
\item[(7)-(9)] The degree, signature and discriminant of the
invariant trace field. Details of minimal polynomials for the fields are available on the web
at {\tt www.ms.unimelb.edu.au/\~{}snap/knotted\_graphs.html}.
\item[(10)] Whether all traces of group elements are algebraic integers.
\item[(11)] Whether the group is arithmetic (after doubling to obtain a finite covolume group).
\end{enumerate}

{
\begin{table}
\begin{center}
\begin{tabular}{c||c|c|c|c|c|c|c|c|c|c}
name & volume &  (K) & sym & a/c & space & deg & sig & disc & int & ar \\ \hline\hline
$2t\_\,2\_\,1 $ & 5.333489567    & 3 & $D_2$& c & $S^3$ & 2 & 0, 1 & $-$7 &Y   & Y\\ \hline
$2t\_\,2\_\,2 $ & 5.333489567    & 3 & $D_6$& c & $L$(3,1) & 2 & 0, 1 & $-$7 &Y& Y\\ \hline
$2t\_\,3\_\,1 $ & 6.354586557    & 4 & $D_2$& c & $\matP^3$ & 3 & 1, 1 & $-$44 &  Y & N\\ \hline
$2t\_\,3\_\,2 $ & 6.354586557    & 4 & $D_2$& c & $L$(4,1) & 3 & 1, 1 & $-$44 & Y & N\\ \hline
$2t\_\,3\_\,3 $ & 6.551743288    & 7 & $D_2$& c & $S^3$ & 3 & 1, 1 & $-$107 & Y & N\\ \hline
$2t\_\,3\_\,4 $ & 6.551743288    & 7 & $D_2$& c & $L$(5,2) & 3 & 1, 1 & $-$107 & Y & N\\ \hline
$2t\_\,4\_\,1 $ & 6.755194816    & 5 & $D_2$& c & $L$(3,1) & 4 & 0, 2 & 2917 & Y & N\\ \hline
$2t\_\,4\_\,2 $ & 6.755194816    & 5 & $D_2$& c & $L$(5,1) & 4 & 0, 2 & 2917 & Y & N\\ \hline
$2t\_\,4\_\,3 $ & 6.927377112    & 11 & $D_2$& c & $S^3$ & 4 & 0, 2 & 1929 & Y & N\\ \hline
$2t\_\,4\_\,4 $ & 6.927377112    & 11 & $D_2$& c & $L$(7,3) & 4 & 0, 2 & 1929 & Y & N \\ \hline
$2t\_\,5\_\,1 $ & 6.952347978    & 6 & $D_2$& c & $L$(4,1) & 5 & 1, 2 & 7684 & Y & N\\ \hline
$2t\_\,5\_\,2 $ & 6.952347978    & 6 & $D_2$& c & $L$(6,1) & 5 & 1, 2 & 7684 & Y & N\\ \hline
$2t\_\,4\_\,5 $ & 6.987763199    & 7 & $D_2$& c & $L$(3,1) & 5 & 1, 2 & 77041 & Y & N\\ \hline
$2t\_\,4\_\,6 $ & 6.987763199    & 7 & $D_2$& c & $L$(7,2) & 5 & 1, 2 & 77041 & Y & N\\ \hline
$2t\_\,4\_\,7 $ & 7.035521457    & 8 & $D_2$& c & $\matP^3$ & 5 & 1, 2 & 5584 & Y & N\\ \hline
$2t\_\,4\_\,8 $ & 7.035521457    & 8 & $D_2$& c & $L$(8,3) & 5 & 1, 2 & 5584 & Y &  N\\ \hline
$2t\_\,5\_\,3 $ & 7.084790037    & 15 & $D_2$& c & $S^3$ & 5 & 1, 2 & 49697 & Y & N\\ \hline
$2t\_\,5\_\,4 $ & 7.084790037    & 15 &    $D_2$& c & $L$(9,2) & 5 & 1, 2 & 49697 & Y & N\\ \hline
$2t\_\,5\_\,5 $ & 7.142157274    & 9 & $D_2$& c & $L$(5,2) & 7 & 1, 3 & $-$123782683 & Y &N\\ \hline
$2t\_\,5\_\,6 $ & 7.142157274    & 9 & $D_2$& c & $L$(9,2) & 7 & 1, 3 & $-$123782683 & Y &N\\ \hline
$2t\_\,5\_\,7 $ & 7.157517365    & 8 & $D_2$& c & $L$(4,1) & 7 &  1, 3 & $-$2369276 & Y & N\\ \hline
$2t\_\,5\_\,8 $ & 7.157517365    & 8 & $D_2$& c & $L$(10,3) & 7 &  1, 3 & $-$2369276 & Y & N \\ \hline
$2t\_\,5\_\,9 $ & 7.175425922    & 9 & $D_2$& c & $L$(3,1) & 7 & 1, 3 & $-$88148831 & Y & N\\ \hline
$2t\_\,5\_\,10 $ & 7.175425922    & 9 & $D_2$& c & $L$(11,3) & 7 & 1, 3 & $-$88148831 & Y & N\\ \hline
$2t\_\,5\_\,11 $ & 7.192635929    & 11 & $D_2$& c & $L$(5,2) & 8 & 0, 4 & 5442461517 & Y &N\\ \hline
$2t\_\,5\_\,12 $ & 7.192635929    & 11 & $D_2$& c & $L$(11,3) & 8 & 0, 4 & 5442461517 & Y &N\\ \hline
$2t\_\,5\_\,13$ & 7.193764490    & 12 & $D_2$& c & $\matP^3$ & 7 & 1, 3 & $-$1523968 & Y &N\\ \hline
$2t\_\,5\_\,14$ & 7.193764490    & 12 & $D_2$& c & $L$(12,5) & 7 & 1, 3 & $-$1523968 & Y &N\\ \hline
$2t\_\,5\_\,15$ & 7.216515907    & 11 & $D_2$& c & $L$(3,1) & 8 & 0, 4 & 3679703653 & Y &N\\ \hline
$2t\_\,5\_\,16$ & 7.216515907    & 11 & $D_2$& c & $L$(13,5) & 8 & 0, 4 & 3679703653 & Y &N \\ \hline
$2t\_\,4\_\,9 $ & 7.327724753    & 4 & $D_2$ & a &  $S^2\times S^1$ & 2 & 0, 1 & $-$4 & Y & Y\\ \hline
$2t\_\,4\_\,10$ & 7.517689896    & 6 & $D_2$& c & $L$(3,1) & 3 & 1, 1&  -104 & Y & N\\ \hline
$2t\_\,4\_\,11$ & 7.706911803    & 5 & $D_2$& c & $S^3$ & 3 & 1, 1& $-$59 & Y & N\\ \hline
$2t\_\,4\_\,12$ & 7.706911803    & 5 & $D_2$& c & $L$(5,1) & 3 & 1, 1& $-$59 & Y & N\\ \hline
$2t\_\,4\_\,13$ & 7.867901276    & 7 & $\matZ_2$& c & $L$(7,2) & 5 &  3, 1 &  -112919 & Y & N\\ \hline
$2t\_\,4\_\,14$ & 7.940579248    & 9 & $D_2$& c & $L$(8,3) & 3  & 1, 1&  -76 & Y & N\\ \hline
$2t\_\,4\_\,15$ & 7.940579248    & 9 & $D_6$& c & $S^3/Q_8$ & 3  & 1, 1&  -76 & Y & N\\ \hline
$2t\_\,4\_\,16$ & 8.000234350    & 4 & $D_2$& c & $\matP^3$ & 2  &0, 1& $-$7 & Y & Y\\ \hline
$2t\_\,5\_\,17$ & 8.087973789    & 5 & $\matZ_2$& c & $S^3$ & 4 & 2, 1& $-$6724 & Y & N\\ \hline
$2t\_\,5\_\,18$ & 8.195703083    & 7 & $\matZ_2$& c & $L$(5,2) & 5 & 1, 2& 65516 & Y & N\\ \hline
$2t\_\,5\_\,19$ & 8.233665208    & 6 & $\matZ_2$& c & $L$(6,1) & 6 & 2, 2& 1738384 & Y& N \\ \hline
$2t\_\,5\_\,20$ & 8.338374585    & 8 & $\matZ_2$& c & $L$(9,2) & 6 & 2, 2& 2463644 & Y & N\\ \hline
$2t\_\,5\_\,21$ & 8.355502146    & 8 & $\matZ_2$& c & $S^3$ & 4 & 0, 2&  3173 & Y & N\\ \hline
$2t\_\,4\_\,17$ & 8.355502146    & 6 & $\matZ_2$& c & $S^3$ & 4 & 0, 2&  3173 & Y & N \\ \hline
$2t\_\,5\_\,22$ & 8.372209945    & 8 & $\matZ_2$& c & $L$(10,3) &7 & 3, 2& 87357184 &Y & N\\ \hline
$2t\_\,5\_\,23$ & 8.388819035    & 10 & $\matZ_2$& c & $L$(4,1) & 5 & 1, 2& 26084 & Y & N
\end{tabular}
\end{center}
\mycap{\label{2t:data:tab:1}
\textbf{Information on hyperbolic $\theta$-graphs up to complexity 5, table 1 of 2.}
Here $Q_8$ denotes the quaternionic group of order $8$ and
$S^3/Q_8$ is the Seifert fibred space $\big(S^2;(2,-1),(2,1),(2,1)\big)$. }
\end{table}}

{
\begin{table}
\begin{center}
\begin{tabular}{c||c|c|c|c|c|c|c|c|c|c}
name & volume &  (K) & sym & a/c & space & deg & sig & disc & int & ar \\ \hline\hline
$2t\_\,5\_\,24$ & 8.403864479    & 10 & $\matZ_2$& c & $L$(11,3) & 7 & 3, 2& 186794473 & Y & N \\ \hline
$2t\_\,5\_\,25$ & 8.487060022    & 8 & $\matZ_2$& c & $L$(9,2) & 8 & 4, 2& 17112324248 & Y & N\\ \hline
$2t\_\,5\_\,26$ & 8.527312899    & 10 & $\matZ_2$& c & $L$(11,3) & 9 &5, 2& 5328053407637 & Y & N\\ \hline
$2t\_\,5\_\,27$ & 8.546347793    & 11 & $\matZ_2$& c & $L$(12,5) & 8 & 4, 2& 2498992192 & Y & N \\ \hline
$2t\_\,5\_\,28$ & 8.565387019    & 12 & $\matZ_2$& c & $L$(13,5) & 9 & 5, 2 & 1944699708173 & Y & N\\ \hline
$2t\_\,5\_\,29$ & 8.612415201    & 1* & $D_2$& c & $L$(4,1) & 4 &  2, 1 & $-$400 & Y & N\\ \hline
$2t\_\,5\_\,30$ & 8.778658803    & 9 & $D_2$& c & $\matP^3$ & 5 & 1, 2 &  15856 & Y& N\\ \hline
$2t\_\,5\_\,31$ & 8.778658803    & 9 & $D_2$& c & $S^3/Q_{12}$  & 5 & 1, 2 &  15856 & Y& N \\ \hline
$2t\_\,5\_\,32$ & 8.793345604    & 7 & $D_2$& c & $S^3$ & 4 & 0, 2 & 257 & Y & N \\ \hline
$2t\_\,5\_\,33$ & 8.806310033    & 8 & $D_2$& c & $L$(8,3)  & 4 & 2, 1 & $-$1968 & Y & N\\ \hline
$2t\_\,5\_\,34$ & 8.908747390    & 11 & $D_2$& c & $L$(3,1) & 5 & 1, 2 & 31048 & Y & N\\ \hline
$2t\_\,4\_\,18$ & 8.929317823    & 6 & $D_2$& c & $S^3$  & 3 & 1, 1 & $-$116 & Y & N\\ \hline
$2t\_\,5\_\,35$ & 8.967360849    & 7 & $D_2$& c & $S^3$ & 4 & 0, 2 & 697 & Y & N\\ \hline
$2t\_\,5\_\,36$ & 8.967360849    & 7 & $D_2$& c & $L$(7,2) & 4 & 0, 2 & 697 & Y & N \\ \hline
$2t\_\,5\_\,37$ & 9.045557688    & 5 & $\matZ_2$& c & $L$(3,1) & 5 & 1, 2 & 73532 & Y   & N\\ \hline
$2t\_\,5\_\,38$ & 9.272866192    & 7 & $\matZ_2$& c & $S^3$ & 6 & 0, 3 & $-$4319731 & Y & N\\ \hline
$2t\_\,5\_\,39$ & 9.353881135    & 7 & $\matZ_2$& c & $L$(3,1) & 6 & 0, 3 & $-$2944468 & Y & N\\ \hline
$2t\_\,5\_\,40$ & 9.437583617    & 9 & $\matZ_2$& c & $\matP^3$ & 4 & 0, 2 & 2312 & Y &N  \\ \hline
$2t\_\,5\_\,41$ & 9.491889687    & 5 & $D_2$& c & $S^3$ & 4 & 0, 2 & 257 & Y & N \\ \hline
$2t\_\,5\_\,42$ & 9.491889687    & 5 & $D_2$& c & $L$(3,1) & 4 & 0, 2 & 257 & Y & N   \\ \hline
$2t\_\,5\_\,43$ & 9.503403931    & 9 & $\matZ_2$& c & $\matP^3$ &4 & 0, 2 & 788 & N & N \\ \hline
$2t\_\,5\_\,44$ & 10.149416064    & 1* & $D_2$& c & $S^2\times S^1$ & 2 & 0, 1 & $-$3 & Y & Y \\ \hline
$2t\_\,5\_\,45$ & 10.396867321    & 6* & $D_3$& c & $S^3$ & 3 & 1, 1 & $-$139 & Y & N \\ \hline
$2t\_\,5\_\,46$ & 10.666979134    & 6 & $\matZ_2$   &   a  &  $S^3$  & 2 & 0, 1 &  -7 & Y & Y\\ \hline
$2t\_\,5\_\,47$ & 10.666979134    & 6 & $\matZ_2$& c & $S^3$ & 2 & 0, 1 & $-$7 & N & N \\ \hline
$2t\_\,5\_\,48$ & 10.666979134    & 5 & $\matZ_2$& c & $L$(3,1) & 2 & 0, 1 &  -7 & Y & Y\\ \hline
$2t\_\,5\_\,49$ & 10.666979134    & 5 & $\matZ_2$& c & $L$(3,1)  & 2 & 0, 1 &  -7 & Y & Y \\
\end{tabular}
\end{center}
\mycap{\label{2t:data:tab:2}
\textbf{Information on hyperbolic $\theta$-graphs up to complexity 5, table 2 of 2.}
Here $Q_{12}$ denotes the generalized quaternionic group of order 12
and $S^3/Q_{12}$ is the Seifert fibred space $\big(S^2;(2,-1),(2,1),(3,1)\big)$.
The Kojima
canonical decompositions of $2t\_\,5\_\,29$ and $2t\_\,5\_\,44$
consist of a cube; the decomposition of
$2t\_\,5\_\,45$  is the union of five tetrahedra and an octahedron.}
\end{table}}

{
\begin{table}
\begin{center}
\begin{tabular}{c||c|c|c|c|c|c|c|c|c|c}
name & volume &  (K) & sym & a/c & space & deg & sig & disc & int & ar \\ \hline\hline
$2h$\_\,1\_\,1 & 3.663862377 & 1 &  $D_4$ &       a & $S^3$ & 2 & 0, 1 &  -4 & Y & Y \\ \hline
$2h$\_\,2\_\,1 & 5.074708032 & 2 &  $D_4$ &       a & $\matP^3$ & 2 & 0, 1 & $-$3 & Y & Y \\ \hline
$2h$\_\,3\_\,1 & 5.875918083 & 3 &  $D_2$ &       c & $L($3,1) & 4 & 0, 2 & 656 & Y & N \\ \hline
$2h$\_\,3\_\,2 & 6.138138789 & 5 &  $D_2$ &       c & $S^3$  & 4 & 0, 2 &  320 & Y & N \\ \hline
$2h$\_\,4\_\,1 & 6.354586557 & 4 &  $D_2$ &       c & $L($4,1) & 3 & 1, 1 &  -44 & Y & N \\ \hline
$2h$\_\,4\_\,2 & 6.559335883 & 5 &  $D_2$ &       c & $L($3,1) & 6 & 0, 3 & $-$382208 & Y & N \\ \hline
$2h$\_\,5\_\,1 & 6.647203159 & 5 &  $D_2$ &       c & $L($5,1) & 6 & 0, 3 & $-$242752 & Y & N \\ \hline
$2h$\_\,4\_\,3 & 6.784755787 & 9 &  $D_2$ &       c & $S^3$ &  6 & 0, 3 & $-$108544 & Y & N \\ \hline
$2h$\_\,4\_\,4 & 6.831770496 & 6 &  $D_2$ &       c & $\matP^3$ & 4 & 0, 2 & 892 & Y & N \\ \hline
$2h$\_\,5\_\,2 & 6.854770090 & 7 &  $D_2$ &       c & $L($5,2) & 8 & 0, 4 & 502248448 & Y & N \\ \hline
$2h$\_\,5\_\,3 & 6.952347978 & 6 &  $D_2$ &       c & $L($4,1) & 5 & 1, 2 & 7684 & Y & N \\ \hline
$2h$\_\,5\_\,4 & 6.969842840 & 5 &  $\matZ_4$ &       a & $L($5,2) & 6 & 0, 3 & $-$179776 & Y & N \\ \hline
$2h$\_\,5\_\,5 & 7.008125009 & 9 &  $D_2$ &       c & $L($5,2) & 10 & 0, 5 & $-$1192884600832 & Y & N \\ \hline
$2h$\_\,5\_\,6 & 7.020614792 & 13 &     $D_2$ &       c & $S^3$ & 8 & 0, 4 & 89276416 & Y & N \\ \hline
$2h$\_\,5\_\,7 & 7.056979121 & 7 &  $D_2$ &       c & $L($3,1) & 10 & 0, 5 & $-$586177642496 & Y & N \\ \hline
$2h$\_\,5\_\,8 & 7.136868364 & 10 &     $D_2$ &       c & $\matP^3$ &  6 & 0, 3 & $-$682736 & Y & N \\ \hline
$2h$\_\,5\_\,9 & 7.146107337 & 9 &  $D_2$ &       c & $L($3,1) & 12 & 0, 6 & 8746362208256 & Y & N \\ \hline
$2h$\_\,3\_\,3 & 7.327724753 & 4 &  $D_2$ &       a & $S^3$ & 2 & 0, 1 & $-$4 & Y & Y \\ \hline
$2h$\_\,4\_\,5 & 7.327724753 & 4 &  $D_2$ &       a & $S^2\times S^1$ & 2 & 0, 1 & $-$4 & Y & Y \\ \hline
$2h$\_\,5\_\,10 & 7.731874058 & 5 &  $\matZ_2$ &       c & $L($4,1)& 6 & 0, 3 & $-$96512 & Y & N \\ \hline
$2h$\_\,5\_\,11 & 8.140719221 & 6 &  $\matZ_2$ &       c & $S^3$ & 6 & 0, 3 & $-$382208 & Y & N \\ \hline
$2h$\_\,5\_\,12 & 8.140719221 & 5 &  $\matZ_2$ &       c & $S^3$ & 6 & 0, 3 & $-$382208 & Y & N \\ \hline
$2h$\_\,4\_\,6 & 8.738570409 & 4 &  $\matZ_2$ &       a & $\matP^3$ & 4 & 0, 2 &144 & Y & N \\ \hline
$2h$\_\,5\_\,13 & 8.997351944 & 3* &     $\matZ_2$ &       c & $S^3$  & 4 & 0, 2 &  784 & Y & N \\ \hline
$2h$\_\,4\_\,7 & 8.997351944 & 4 &  $\{{\rm id}\}$ &     c & $S^3$ & 4 & 0, 2 &  784 & Y & N  \\ \hline
$2h$\_\,4\_\,8 & 8.997351944 & 4 &  $\matZ_2$ &       c & $L($3,1) & 4 & 0, 2 &  784 & Y & N \\ \hline
$2h$\_\,5\_\,14 & 9.539780459 & 5 &  $\{{\rm id}\}$ &     c & $L($3,1) & 4 & 0, 2& 656 & Y & N \\ \hline
$2h$\_\,5\_\,15 & 9.539780459 & 5 &  $D_2$ &       c & $S^3$  & 4 & 0, 2& 656 & Y & N \\ \hline
$2h$\_\,5\_\,16 & 9.592627932 & 6 &  $D_2$ &       c & $\matP^3$ & 4 & 0, 2& 1436 & Y & N \\ \hline
$2h$\_\,5\_\,17 & 9.802001166 & 5 &  $\{{\rm id}\}$ &     c & $S^3$ & 4 & 0, 2 & 320 & N & N \\ \hline
$2h$\_\,5\_\,18 & 9.876829057 & 5 &  $\matZ_2$ &       c & $S^3$ & 6 & 0, 3 & $-$239168 & Y & N \\ \hline
$2h$\_\,5\_\,19 & 10.018448934 & 5 &  $\{{\rm id}\}$ &     c & $\matP^3$ & 6 & 0, 3 & $-$30976 & N & N \\ \hline
$2h$\_\,5\_\,20 & 10.018448934 & 5 &  $\{{\rm id}\}$ &     c & $L($4,1) & 6 & 0, 3 & $-$30976 & Y & N  \\ \hline
$2h$\_\,5\_\,21 & 10.018448934 & 5 &  $\matZ_2$ &       c & $L($4,1) & 6 & 0, 3 & $-$30976 & Y & N  \\ \hline
$2h$\_\,5\_\,22 & 10.069070958 & 7 &  $\matZ_2$ &       c & $\matP^3$ & 4 & 0, 2 & 1384 & Y & N \\ \hline
$2h$\_\,5\_\,23 & 10.149416064 & 4* &     $\matZ_2$ &       c & $S^2\times S^1$ & 2 & 0, 1 & $-$3  & Y & Y \\ \hline
$2h$\_\,5\_\,24 & 10.215605665 & 5 &  $\{{\rm id}\}$ &     c & $S^3$ & 6 & 0, 3 & $-$732736 & N    & N \\ \hline
$2h$\_\,5\_\,25 & 10.215605665 & 5 &  $\{{\rm id}\}$ &     c & $L($5,2) & 6 & 0, 3 & $-$732736 & Y & N \\ \hline
$2h$\_\,5\_\,26 & 10.215605665 & 5 &  $\matZ_2$ &       c & $L($5,2) & 6 & 0, 3 & $-$732736 & Y & N \\ \hline
$2h$\_\,5\_\,27 & 10.408197599 & 5 &  $\{{\rm id}\}$ &     c & $\matP^3$ & 4 & 0, 2 & 441 & N &N \\
\end{tabular}
\end{center}
\mycap{\label{2h:data:tab:1}
\textbf{Information on hyperbolic handcuff graphs up to complexity 5.}
The Kojima canonical decomposition of $2h$\_\,5\_\,13 is the union of a tetrahedron and two pyramids with square base; the decomposition for
$2h$\_\,5\_\,23 is the union of two tetrahedra and two pyramids with square base.}
\end{table}}

{
\begin{table}
\begin{center}
\begin{tabular}{c||c|c|c|c|c|c|c|c|c|c}
name & volume &  (K) & sym & a/c & space & deg & sig & disc & int & ar \\ \hline\hline
$4a$\_\,2\_\,1 & 7.327724753 & 2 &  $\matZ_2\times O$ & a & $S^3$  & 2 & 0, 1 & $-$4 & Y & Y\\ \hline
$4a$\_\,5\_\,1 & 11.751836165 & 6 &  $D_4$ &   c & $S^3$  & 4 & 0, 2 & 656 & Y & N \\ \hline
$4a$\_\,5\_\,2 & 12.661214320 & 5 &  $\matZ_2$ &   c & $S^3$ & 4 & 0, 2 & 784 & Y & N \\ \hline\hline
$4b$\_\,4\_\,1 & 10.149416064 & 4 &     $\matZ_2\times D_4$ &    a & $S^3$  & 2 & 0, 1 &  -3 & Y & Y \\ \hline\hline
$4c$\_\,4\_\,1 & 10.991587130 & 4 &     $D_2$ &   a & $S^3$ & 2 & 0, 1 &  -4 & Y & Y \\
\end{tabular}
\end{center}
\mycap{\label{o:data:tab}
\textbf{Information on hyperbolic 4-vertex graphs
up to complexity 5.} Here $O$ denotes the group of
orientation-preserving symmetries of the regular octahedron,
isomorphic to the full group of symmetries of the regular
tetrahedron.}
\end{table}}

\subsection{Table of knots} \label{table:knots:subsection}
As already mentioned, we have classified hyperbolic knots only up to
complexity 4, finding 5 of them. The table containing their description
differs from the previous ones only in that the third column gives the number of cells in the
Epstein-Penner~\cite{EpPe} canonical decomposition (the Kojima decomposition is not defined).
We also provide an additional table showing
the name of each knot complement in the SnapPea census~\cite{CaHiWe}, and
either the name of the knot in~\cite{Rolfsen} (for
the knots in $S^3$) or the surgery coefficients on one of the
components of the Whitehead link ($5_1^2$ in~\cite{Rolfsen})
yielding the knot.

As shown in the introduction and in Section~\ref{non:section}
below, there are many $(0,1,2)$-irreducible knots in complexity up
to $3$, and most of them are not hyperbolic: this phenomenon can
be understood using spines, see Proposition~\ref{complement:prop}.

{
\begin{table}
\begin{center}
\begin{tabular}{c||c|c|c|c|c|c|c|c|c|c}
name & volume &  (EP) & sym & a/c & space & deg & sig & disc & int & ar \\ \hline\hline
0\_\,3\_\,1 & 2.029883213 & 2 & $D_4$ & a  & $S^3$      & 2 & 0,1& $-$3 & Y & Y\\ \hline
0\_\,4\_\,1 & 2.029883213 & 2 & $D_2$ & c  & $L(5,1)$   & 2 & 0,1& $-$3 & Y & Y \\ \hline
0\_\,4\_\,2 & 2.568970601 & 4 & $D_2$ & c  & $L(3,1)$   & 3 & 1,1 & $-$59 &  Y & N \\ \hline
0\_\,4\_\,3 & 2.666744783 & 3 & $D_2$ & c  & $\matP^3$  & 2 & 0,1& $-$7 &Y & Y\\ \hline
0\_\,4\_\,4 & 2.828122088 & 4 & $D_2$ & c  & $S^3$      & 3 & 1,1 & $-$59 &  Y & N
\end{tabular}
\end{center}
\mycap{\label{k:data:tab}
\textbf{Information on hyperbolic knots up to complexity 4.}}
\end{table}}

{
\begin{table}
\begin{center}
\begin{tabular}{c||c|c}
name & in~\cite{CaHiWe} & in~\cite{Rolfsen} \\ \hline\hline
0\_\,3\_\,1 & $m$004 & $4_1$ \\ \hline
0\_\,4\_\,1 & $m$003 & $5_1^2(-5,1)$ \\ \hline
0\_\,4\_\,2 & $m$007 & $5_1^2(-3,2)$ \\ \hline
0\_\,4\_\,3 & $m$009 & $5_1^2(2,1)$ \\ \hline
0\_\,4\_\,4 & $m$015 & $5_2$
\end{tabular}
\end{center}
\mycap{\label{k:names:tab}
\textbf{Other names for hyperbolic knots up to complexity 4.}}
\end{table}}

\subsection{Compact totally geodesic boundary} \label{geodesic:subsection}
The 3 graphs referred to in Proposition~\ref{cpt:bd:prop} are
$2t\_\,5\_\,45$, $2t\_\,5\_\,46$ and $2t\_\,5\_\,47$ in
Table~\ref{2t:data:tab:2}; these are shown in
Figure~\ref{geod_bound:fig}. (In particular, Thurston's knotted
$Y$~\cite[pp.~133-137]{bible} is $2t\_\,5\_\,45$.)
Their hyperbolic structures were constructed using \emph{Orb}. They
all have the lowest possible volume
($\approx 6.45199027$) for hyperbolic $3$-manifolds with genus 2
boundary (see~\cite{KoMi}), but they can be distinguished by their
Kojima decompositions or symmetry groups. All the other graphs
were shown not to have such a structure by studying spines for
their complements constructed as in the proof of
Proposition~\ref{complement:prop}. In all but two cases, this
produced a spine for the complement of complexity having less than $2$
vertices, hence
the complement has no hyperbolic structure with geodesic boundary
by Remark~\ref{complement:rem}.  For the
two remaining cases, we found a spine having 2 vertices  but not
dual to a triangulation. It again follows that these manifolds are
not hyperbolic with geodesic boundary, because a minimal simple
spine of a hyperbolic manifold is always dual to a
triangulation~\cite{Matv:book}.

\section{Irreducible non-hyperbolic graphs}\label{non:section}

This section is devoted to the description of the
$(0,1,2)$-irreducible but non-hyperbolic graphs we have found
in our census, including the proof that indeed they have these
properties.

\subsection{Knots and links}\label{non:knots:subsection}
As already stated in the introduction, we have shown that if a
graph $(M,G)$ with $c(M,G)\leqslant 4$ is $(0,1,2)$-irreducible
but non-hyperbolic then $G$ has no vertices. More precisely, $G$
is either empty, or a knot, or a two-component link. Since this
paper is chiefly devoted to the understanding of graphs
\emph{with} vertices, we will only very briefly describe our
discoveries for the case without vertices.  In particular, we will
not refer to the case of empty $G$ (\emph{i.e.}, to the case of
manifolds), addressing the reader to~\cite{Matv:book}, and we will
describe the following non-hyperbolic knots and links:
\begin{itemize}
\item up to complexity $2$, in general manifolds;
\item in complexity $3$, in $S^3$.
\end{itemize}
To proceed we will introduce some general machinery.

\subsection{Torus knots in lens spaces}\label{torus_knot_sect}
Consider the solid torus $\matT$ and the basis of $H_1(\partial\matT)$
given by a longitude $\lambda$ and a meridian $\mu$. These elements are
characterized up to symmetries of $\matT$ by the property that the restriction
to $\langle\lambda\rangle$ of the map $i_*:H_1(\partial\matT)\to H_1(\matT)$
is surjective, while $\langle\mu\rangle$ is the kernel of this map.

For coprime $\ell,m\in\matZ$ we will denote by $K(\ell,m)$ a simple closed curve
on $\partial\matT$ (unique up to isotopy) representing
$\ell\cdot\lambda+m\cdot\mu$ in $H_1(\partial\matT)$.
For $n\geqslant 2$ we will also denote by $K(n\cdot \ell,n\cdot m)$ the union of
$n$ parallel copies of $K(\ell,m)$.

We will assume from now on that the lens space $L(p,q)$ is obtained
from $\matT$ by Dehn filling along $K(p,q)$. Therefore any $K(\ell,m)$
can be viewed as a {\em torus knot} on the Heegaard torus $\partial\matT$ in $L(p,q)$.
An easy application of the Seifert-Van Kampen
theorem implies the following:

\begin{prop}\label{pi1:torus:prop}
For $\ell,m$ coprime integers,
$\pi_1\big(L(p,q)\setminus K(\ell,m)\big)\cong\langle x,y|\ x^a=y^b\rangle$
with $a=| \ell |$ and $b=| p m-q \ell |$.
\end{prop}

\begin{rem}
\emph{The curves $K(\ell,m)$ and $K(m,\ell)$ coincide as knots
in $L(1,0)=S^3$. For instance $K(2,3)$ and $K(3,2)$ are equivalent trefoil knots in $L(1,0)=S^3$.
This is of course
consistent with the computation of the fundamental group.}
\end{rem}

\begin{prop}\label{irred:torus:prop}
If $(\ell,m)=(p,q)=1$ then $\big(L(p,q),K(\ell,m)\big)$ is a $(0,1,2)$-irreducible pair
except in the following cases:
\begin{itemize}
\item $\ell=0$ or $p m-q \ell=0$, and $q\neq 0$ (\emph{i.e.}, $L(p,q)\neq S^3$);
\item $|\ell|\leqslant 2$ and $p=0$ (\emph{i.e.}, $L(p,q)=S^2\times S^1$).
\end{itemize}
\end{prop}

\begin{proof}
If $\ell=0$ or $p m-q \ell=0$ then $K:=K(\ell,m)$ bounds a
meridian disc of either $\matT$ or the complementary solid torus attached to $\partial \matT$.
Therefore $K$ is the unknot, and the pair is not $0$-irreducible
when $L(p,q)\neq S^3$. If $L(p,q)=S^2\times S^1$, the knot $K$
intersects the sphere $S^2\times\{pt\}$ in $|\ell|$ points.
Therefore if $|\ell|\leqslant 2$ the pair is not
$|\ell|$-irreducible.

Conversely, let us assume that there exists an essential sphere $S$ in $L(p,q)$
meeting $K:=K(\ell,m)$ transversely in $t\leqslant 2$ points. Suppose first that $t=0$.
If $|\ell|$ and $|pm-q\ell|$ are non-zero,
the complement of $K$ in $L(p,q)$ has a Seifert fibration over the disc with two
singular fibers of orders $|\ell|$ and $|pm-q\ell|$: such a manifold is irreducible,
so $S$ cannot be essential,
a contradiction. So either $\ell=0$ or $pm-q\ell=0$, which implies that $K$ is the unknot
in one of the solid tori and $S$ is the boundary of a ball containing $K$.
Since $S$ is essential it follows that
$M\ne S^3$, namely $q\neq 0$. (This argument shows in particular
that when $L(p,q)=S^2\times S^1$ (\emph{i.e.}, $p=0$), the pair
$(L(p,q),K)$ is 0-reducible only for $\ell=0$.)

Suppose now $t\ne 0$ and assume, after an isotopy, that $S$ is transverse to
the Heegaard torus $\partial\matT$. Considering this transverse intersection on $S$
we see that there must be at least two innermost discs. Moreover any innermost
disc belongs to one of the following types:

\begin{itemize}
\item[(I)] Its boundary is inessential on $\partial\matT$ and disjoint from $K$;
\item[(II)] Its boundary is inessential on $\partial\matT$ and meets $K$ transversely in two points;
\item[(III)] It is a meridian disc of either $\matT$ or of the complementary solid torus.
\end{itemize}

Discs of type (I) can be removed by an isotopy.
If there is a disc of type (II)
then doing surgery close to it we can replace $S$ by an essential sphere disjoint from
$K$, so we are led back to the case $t=0$. Therefore we can assume all the discs are of type (III).
If $\ell=0$ or $pm-q\ell = 0$,
we can again reduce to the case $t=0$. So we can assume that
all the innermost discs
meet $K$, which easily implies that there are only two of them, either sharing
their boundary or separated by an annulus. In the first case we see that
$M=S^2\times S^1$ (\emph{i.e.}, $p=0$) and $1\leqslant |\ell|\leqslant 2$.
In the second case we deduce that $S$ is actually inessential,
which is absurd. This concludes the proof.
\end{proof}

\subsection{Layered triangulations}
A {\em layered triangulation} (see ~\cite{layered}) of a lens space $L(p,q)$  is
constructed as follows. We start with a solid torus triangulated using one tetrahedron
as in Figure~\ref{1tet_torus}. The boundary torus is triangulated by 2 triangles,
3 edges and 1 vertex. A change of the triangulation on the boundary
by a diagonal exchange move (``flip'') can be realized by adding one tetrahedron. After
a series of these moves, the resulting triangulation can be closed up by adding
another 1-tetrahedron triangulation of a solid torus to produce a lens space.

    \begin{figure}
    \begin{center}
\includegraphics[scale=1.0]{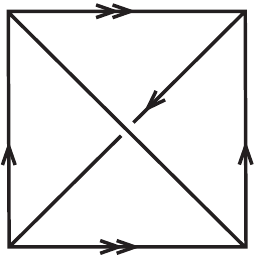}
    \mycap{\label{1tet_torus} A 1-tetrahedron triangulation of the solid torus.
    The back two triangles are glued together to form a M\"obius strip.
    The front two triangles form the boundary torus.
    }
    \end{center}
    \end{figure}

Such a layered triangulation of $L(p,q)$ with
one vertex and one marked edge always gives rise to some torus
knot $K(\ell,m)\subset L(p,q)$. Using the Farey tessellation
of hyperbolic plane $\H^2$ we will now show the converse, namely for
every torus knot $\big(L(p,q), K(\ell,m)\big)$ we will construct
a layered triangulation.

Recall that the Farey tessellation of $\H^2$ is constructed in the
half-plane model by joining with a geodesic every pair $(p/q,r/s)$
of rational ideal points in $\mathbb Q
\cup\{\infty\}\subset\partial\H^2$ 
where $p,q,r,s$ are integers with $ps-qr= \pm 1$.
After fixing some basis for $H_1(T)$, every
slope (\emph{i.e.}, unoriented essential simple closed curve) on a
torus $T$ is represented by a rational number $p/q\in\partial \H^2$,
and two such numbers are connected by an edge of the tessellation
when they have geometric intersection number $1$. 

Every triangle of the tessellation represents three slopes with
pairwise intersection $1$, and hence a $1$-vertex triangulation of
$T$. Dually, they represent a $\theta$-graph in $T$ as in Figure
\ref{Mobius_triplet}-(1-top). Moreover, every edge of the
tessellation represents a flip relating the $\theta$-graphs of $T$
corresponding to the adjacent triangles as in
Figure~\ref{Mobius_triplet}-(2,3).

A layered triangulation of a lens space $L(p,q)$ is easily
encoded via a \emph{path of triangles} of the tessellation
connecting the rational numbers $0/1$ and $p/q$, \emph{i.e.}, a
sequence $f_1,\ldots, f_k$ of $k\geqslant 4$ triangles such that
$f_{i-1}$ and $f_i$ share an edge for $i=2,\ldots,k$, the vertex of
$f_1$ disjoint from $f_2$ is $0/1$, and the vertex of $f_k$ disjoint
from $f_{k-1}$ is $p/q$. The path need not to be injective,
\emph{i.e.}, there may be repetitions. Such a path is similar to the
one defined in~\cite{layered, MarPet} for layered solid tori. It
determines a layered triangulation of $L(p,q)$ with $k-3$
tetrahedra, $k-2$ edges and $1$ vertex, as described in Fig
\ref{Mobius_triplet}.

The $k-2$ edges of the layered triangulation become
torus knots, and they correspond to all the slopes $\ell/m$ contained in
some $f_i$ except $0/1$ and $p/q$. (There are $k$ different such slopes, but the two
in $f_1$ different from $0/1$ give isotopic links in $L(p,q)$, and in fact
the same edge in the layered triangulation, and similarly for
the two slopes in $f_k$ different from $p/q$, whence the number $k-2$).
See Figure~\ref{layered_tori} for some examples.

    \begin{figure}
    \begin{center}
    \includegraphics[scale=1.0]{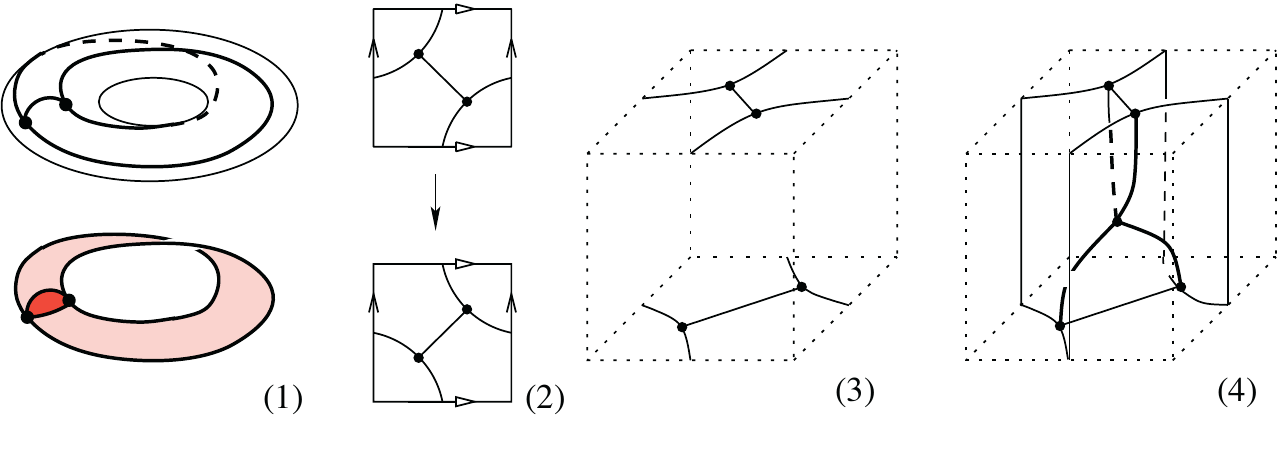}
\mycap{\label{Mobius_triplet} A path of triangles $f_1,\ldots,f_k$
in the Farey tessellation determines a layered triangulation of a
lens space, as follows. We describe the dual special spine. The
vertices of $f_2$ are $1$, $2$, $\infty$ and they determine the
$\theta$-graph in $\partial\matT$ shown in (1-top). We take a
portion of spine, made of a M\"obius strip and one disc, bounded by 
this $\theta$-graph (1-bottom). Each step from $f_i$ to
$f_{i+1}$ for $2 \leqslant i \leqslant k-2$
corresponds to a diagonal flip of the $\theta$-graph (2,3)
which expands the portion of spine by creating a vertex (4).
Finally, we close the spine at $f_{k-1}$ by adding an analogous
M\"obius strip for the other Heegaard torus. There are $k-3$ flips
and hence $k-3$ vertices in the spine.}
    \end{center}
    \end{figure}

    \begin{figure}
    \begin{center}
    \includegraphics[scale=0.9]{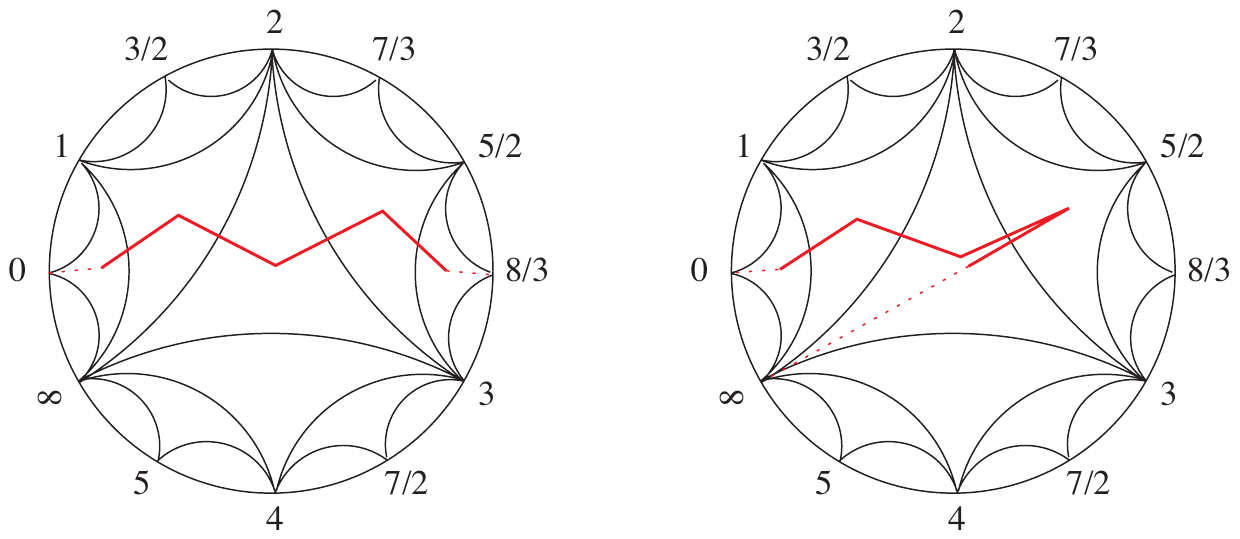}
\mycap{\label{layered_tori} Two paths of triangles. The first gives
a triangulation of $L(8,3)$ containing the torus knots $K(1,0)$ and
$K(2,1)$, and other torus knots equivalent to these. The second path
is not injective and gives a triangulation of $L(1,0)$ containing
$K(5,2)$, \emph{i.e.} the $(5,2)$ torus knot in $S^3$. Both
triangulations contain $5-3=2$ tetrahedra.}
    \end{center}
    \end{figure}

Let then $\lambda(\ell,m,p,q)$ be the length of the shortest
path of triangles from $0/1$ to $p/q$ which contains $\ell/m$. By
what just said, we have:
$$c\big(L(p,q),K(\ell,m)\big)\leqslant \max\big\{\lambda(\ell,m,p,q)-3,0\big\}.$$
It was conjectured in~\cite{Matv:AAM} that every
$L(p,q)=\big(L(p,q),\emptyset\big)$ with $c\neq 0$ has
a minimal triangulation which is layered,
namely that $c\big(L(p,q)\big)=\max\big\{\lambda(p,q)-3,0\big\}$, where
$\lambda(p,q)$ is the length of the shortest
path of triangles from $0/1$ to $p/q$. We now propose the following extension:

\begin{conj} \label{torus:conj} The complexity of a $(0,1,2)$-irreducible torus knot
in a lens space is
$$c\big(L(p,q),K(\ell,m)\big)= \max\big\{\lambda(\ell,m,p,q)-3,0\big\}.$$
\end{conj}
As the census in Table~\ref{non:knot:data:tab} shows, the conjecture holds for
complexity up to $2$.

\newpage
\subsection{Non-hyperbolic knots and links}
The non-hyperbolic knots and links up to complexity 2, and those
having complexity 3 contained in $S^3$, are described in
Table~\ref{non:knot:data:tab}. They are all torus links in lens
spaces, except for a knot in the elliptic Seifert space $S^3/{Q_8}$,
whose exterior is the twisted interval bundle over the Klein bottle.
This pair is pictured in Figure~\ref{Q8:knot:fig}.

Note that $L(7,2)$ is the only lens
space in the table not admitting a symmetry switching the two
cores of the Heegaard solid tori, and that both these cores appear
in the list.

\begin{table}[h]
\begin{center}
\begin{tabular}{c||c|c|l}
$c$ & type & space & description of knot or link \\ \hline\hline
0 & knot & $S^3$ & $K(1,0)=$ unknot\\ \hline
0 & knot & $\matP^3$ & $K(1,0)=$ core of Heegaard torus\\ \hline
0 & knot & $L(3,1)$ & $K(1,0)=$ core of Heegaard torus\\ \hline\hline

1 & knot & $S^3$ & $K(3,2)=$ trefoil\\ \hline
1 & link & $S^3$ & $K(2,2)=$ Hopf link\\ \hline
1 & knot & $L(4,1)$ & $K(1,0)=$ core of Heegaard torus\\ \hline
1 & knot & $L(5,2)$ & $K(1,0)=$ core of Heegaard torus\\ \hline\hline

2 & knot & $S^3$ & $K(5,2)= 5_1$~\cite{Rolfsen}\\ \hline 2 & knot
& $L(5,1)$ & $K(1,0)=$ core of Heegaard torus\\ \hline 2 & knot &
$L(7,2)$ & $K(1,0)=$ core of one Heegaard torus\\ \hline 2 & knot
& $L(7,2)$ & $K(3,1)=$ core of other Heegaard torus\\ \hline 2 &
knot & $L(8,3)$ & $K(1,0)=$ core of Heegaard torus\\ \hline 2 &
knot & $L(5,1)$ & $K(2,1)$ \\ \hline 2 & knot & $L(7,2)$ &
$K(2,1)$ \\ \hline 2 & knot & $L(8,3)$ & $K(2,1)$ \\ \hline 2 &
knot & $S^2\times S^1$ & $K(3,1)$ \\ \hline 2 & knot & $L(3,1)$ &
$K(3,2)$ \\ \hline 2 & knot & $\matP^3$ & $K(4,1)$ \\ \hline 2 &
link & $\matP^3$ & $K(2,2)=$ union of cores of Heegaard tori\\
\hline 2 & knot & $S^3/Q_8$ & singular fibre of
$\big(S^2;(2,-1),(2,1),(2,1)\big)$ \\ \hline\hline

3 & knot & $S^3$ & $K(4,3)=8_{19}$~\cite{Rolfsen} \\ \hline 3 &
knot & $S^3$ & $K(5,3)=10_{123}$~\cite{Rolfsen} \\ \hline 3 & knot
& $S^3$ & $K(7,2)=7_1$~\cite{Rolfsen}\\ \hline
3 & link & $S^3$ & $K(4,2)=4^2_1$~\cite{Rolfsen}\\
\end{tabular}
\end{center}
\mycap{\label{non:knot:data:tab}
\textbf{Information on non-hyperbolic knots and links.} In complexity 3 only knots
and links in the 3-sphere are described. In the description of torus knots, we set $S^3=L(1,0)$,
$\matP^3=L(2,1)$, and $S^2\times S^1=L(0,1)$.}
\end{table}

    \begin{figure}[h]
    \begin{center}
    \includegraphics[scale=0.7]{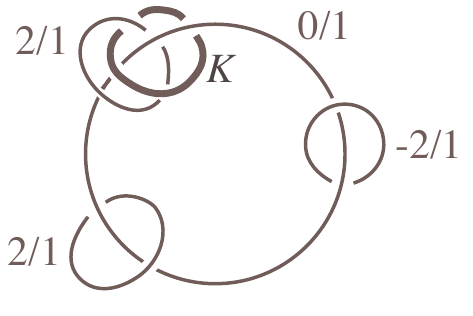}
    \mycap{\label{Q8:knot:fig} A surgery presentation of the pair $(M,K)$ where
    $M=S^3/Q_8=\big(S^2;(2,-1),(2,1),(2,1)\big)$ and $K$ is a singular fibre of the fibration. }
    \end{center}
    \end{figure}

\subsection{$\theta$-graphs with Klein bottles}
In complexity 5 we have only investigated pairs $(M,G)$ where $G$
is non-empty and all its components have vertices. As mentioned above,
we have found here 5 very interesting pairs, where $G$ is a $\theta$-graph
and the pair $(M,G)$ is $(0,1,2)$-irreducible,  but non-hyperbolic since $M \setminus G$
contains an embedded Klein bottle, so it is not atoroidal.

\begin{prop} \label{Klein:irred:prop}
There are five $(0,1,2)$-irreducible non-hyperbolic pairs $(M,G)$
such that $c(M,G)=5$ and $G$ has no knot component. They are described as follows:
\begin{itemize}
\item[(i)] Let $\matK$ be the twisted interval bundle over the Klein bottle.
\item[(ii)] Let $(\matT,\theta)$ be the solid torus with the embedded $\theta$-graph
shown in Figure~\ref{theta_in_torus}.
\item[(iii)]
Then $(M,G)$ is obtained by gluing $\matK$ to $(\matT,\theta)$ so
that $M$ is one of the manifolds $S^2\times S^1$, $S^3/{Q_8}$, $L(8,3)$,
$L(4,1)$, or $\mathbb {RP}^3\#\mathbb{RP}^3$.
\end{itemize}
\end{prop}

    \begin{figure}[h]
    \begin{center}
    \includegraphics[scale=0.9]{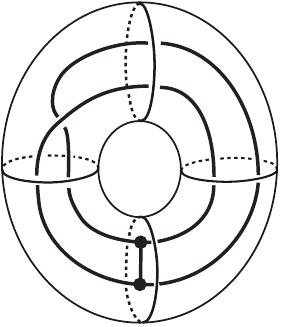}
    \mycap{\label{theta_in_torus} The theta graph $\theta$ in the solid torus $\matT$.}
    \end{center}
    \end{figure}

This result was proved as follows.
We first analyzed the triangulations of the $5$
pairs $(M,G)$ produced by our Haskell code on which \emph{Orb} failed to construct a hyperbolic
structure. This allowed us to show that the  $5$ pairs are those described in points (i)-(iii)
of the statement, whence to see that they are not hyperbolic.
We then proved that they are indeed (0,1,2)-irreducible by classical
topological techniques, the key point being that
a compressing disc of $(\matT,\theta)$ must intersect
$\theta$ in at least two points.

Here are the details of the argument. Suppose there is a sphere $S$
intersecting $G$ transversely
in at most $2$ points, and isotope $S$ to
minimize its intersection with $\partial\matT$.
Now consider an innermost disc $D$ on $S$
bounded by a simple closed curve in $S \cap\partial\matT$. Since there is no
compressing disc in $\matK$, such a disc must be a
compressing disc in $\matT$, so it must intersect $\theta$
at least twice. But if $S \cap \partial\matT \ne \emptyset$ then there are at least
two innermost discs on $S$, whence $S\cap G$ contains at least 4
points, which is impossible. This shows that $S$ is disjoint from $\partial\matT$,
so it is contained either in $\matK$ or in $\matT$. However $\matK$ is irreducible,
and $(\matT,\theta)$ is $(0,1,2)$-irreducible (in
fact, it is easy to see that it is hyperbolic with parabolic
meridians). Therefore $S$ must bound a trivial ball in $(M,G)$.

\section{Figures}\label{fig:section}
This section contains pictures of the hyperbolic graphs up to
complexity 4, given in the form of a surgery description when the
underlying space is not $S^3$. For each graph, we give the name
and the volume of the hyperbolic structure with parabolic
meridians.

The figures were produced using \emph{Orb}~\cite{orb} and the census of
knotted graphs in~\cite{CHHSS}.  Most of the graphs in $S^3$
occurred in~\cite{CHHSS}; the graphs not in $S^3$ generally arose
as Dehn surgeries on knot components of disconnected graphs
in~\cite{CHHSS}.  There were a couple of remaining examples which
were constructed by hand. In all cases, we used \emph{Orb} to  identify
the graphs by matching triangulations.

    \begin{figure}
    \begin{center}
    \includegraphics[scale=1.0]{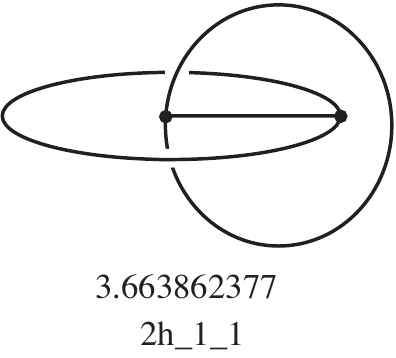}
    \mycap{\label{c=1:fig} Complexity 1.}
    \end{center}
    \end{figure}

     \begin{figure}
    \begin{center}
    \includegraphics[scale=1.0]{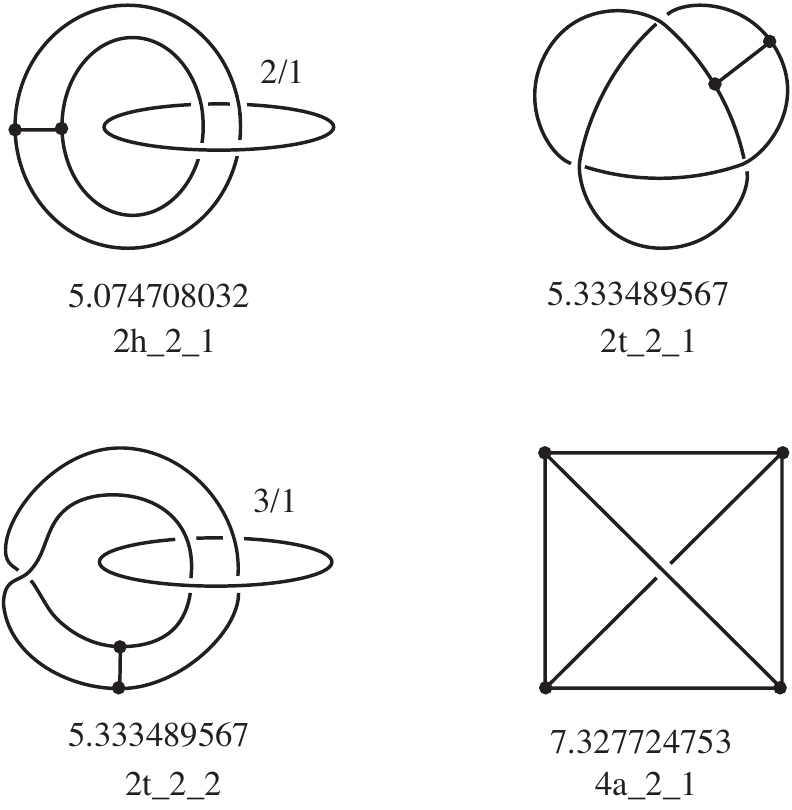}
    \mycap{\label{c=2:fig} Complexity 2.}
    \end{center}
    \end{figure}

\ 
\vfill\eject

    \begin{figure}
    \begin{center}
    \includegraphics[scale=1.0]{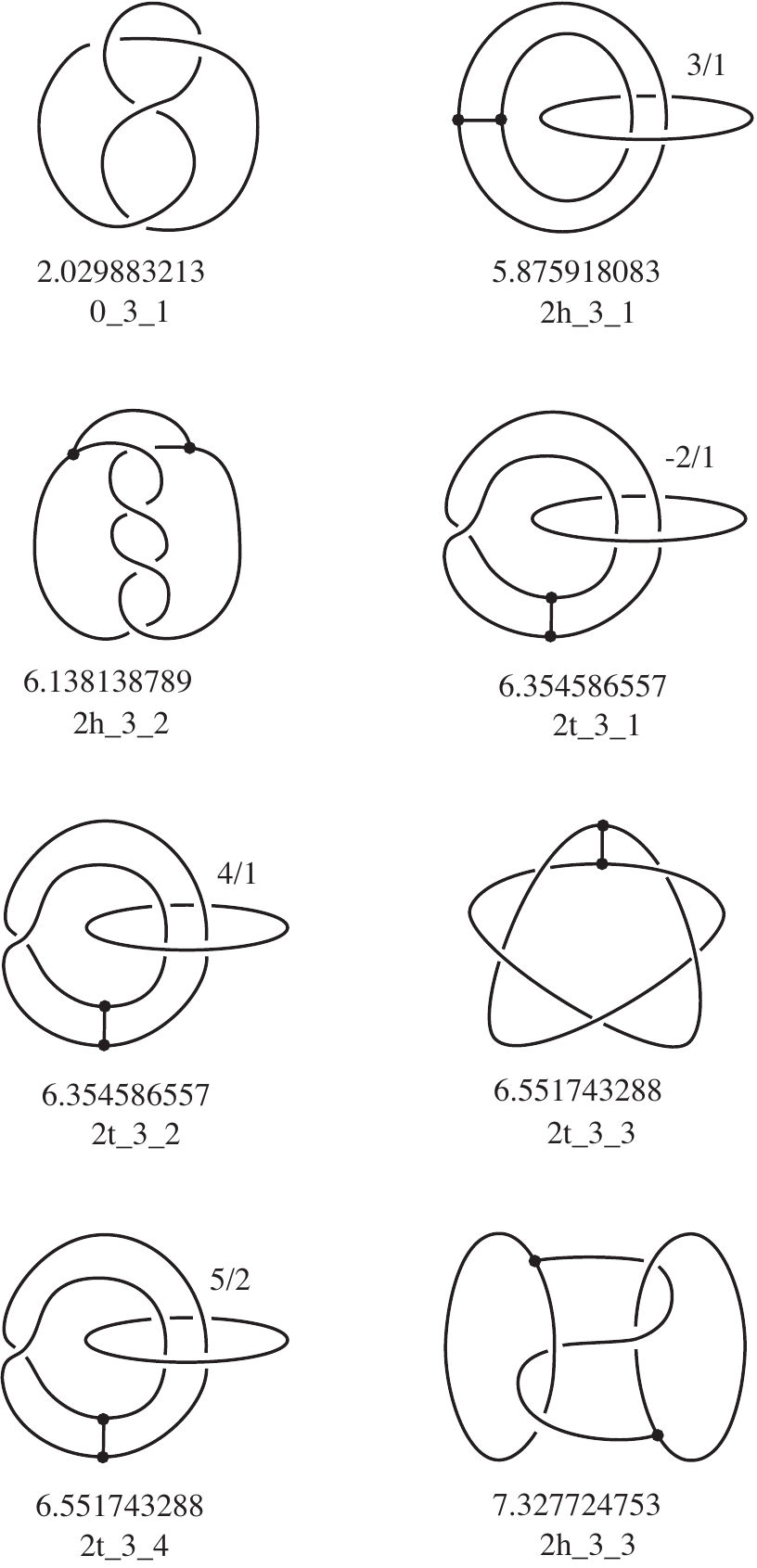}
    \mycap{\label{c=3:fig} Complexity 3.}
    \end{center}
    \end{figure}
    
\ 
\vfill\eject

    \begin{figure}
    \begin{center}
    \includegraphics[scale=1.0]{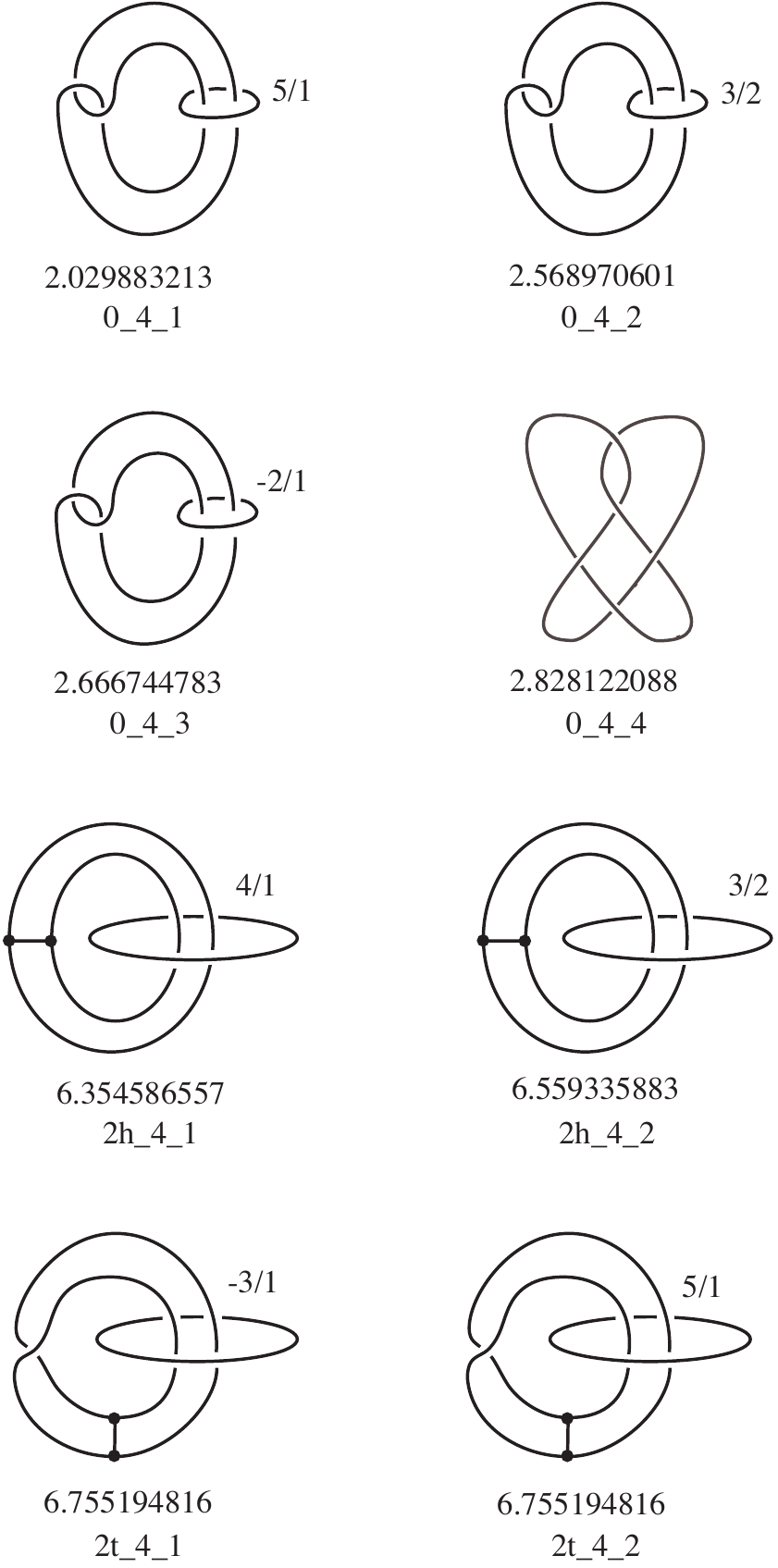}
    \mycap{\label{c=4:fig1} Complexity 4, part 1 of 4.}
    \end{center}
    \end{figure}

\ 
\vfill\eject

    \begin{figure}
    \begin{center}
    \includegraphics[scale=1.0]{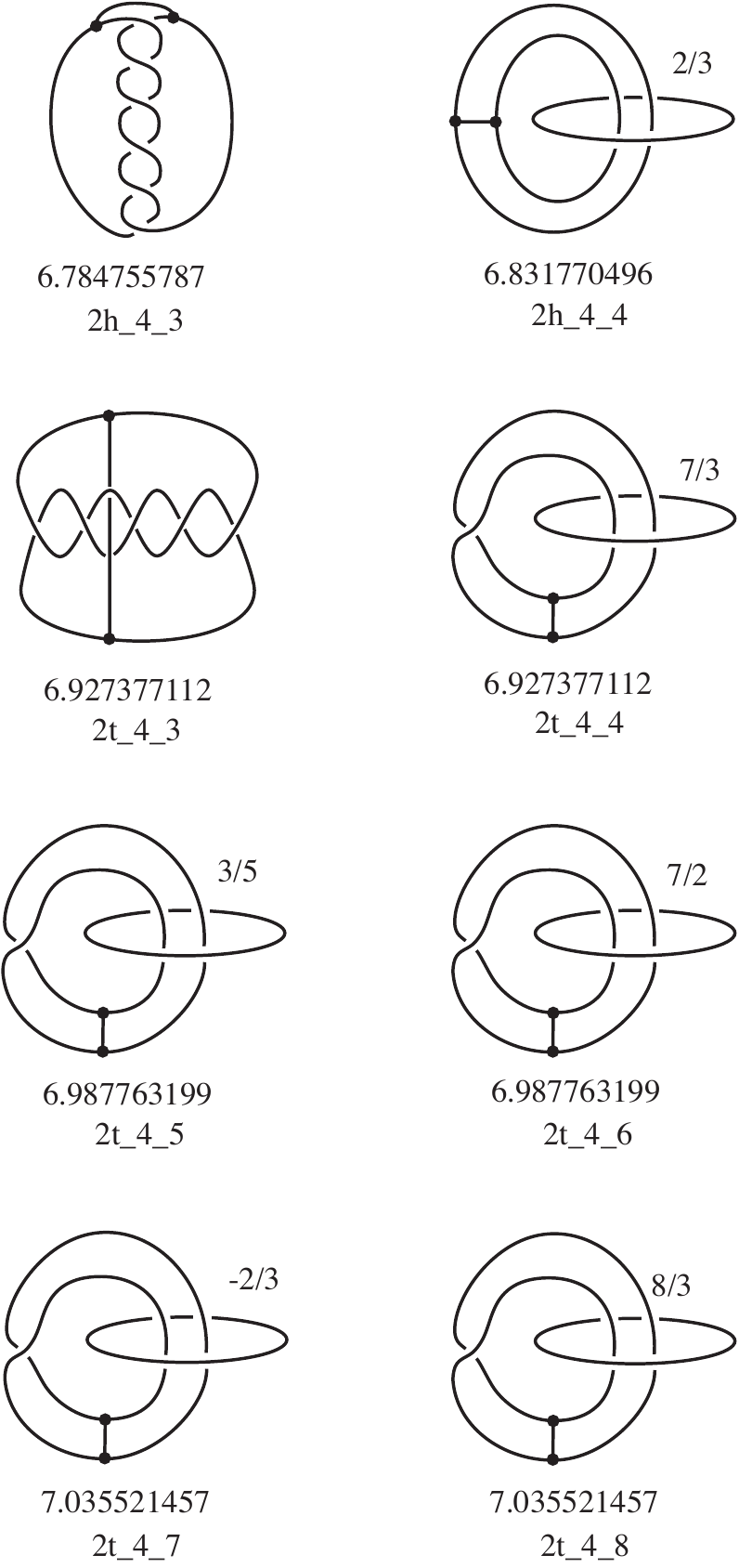}
    \mycap{\label{c=4:fig2} Complexity 4, part 2 of 4.}
    \end{center}
    \end{figure}

\ 
\vfill\eject

    \begin{figure}
    \begin{center}
    \includegraphics[scale=1.0]{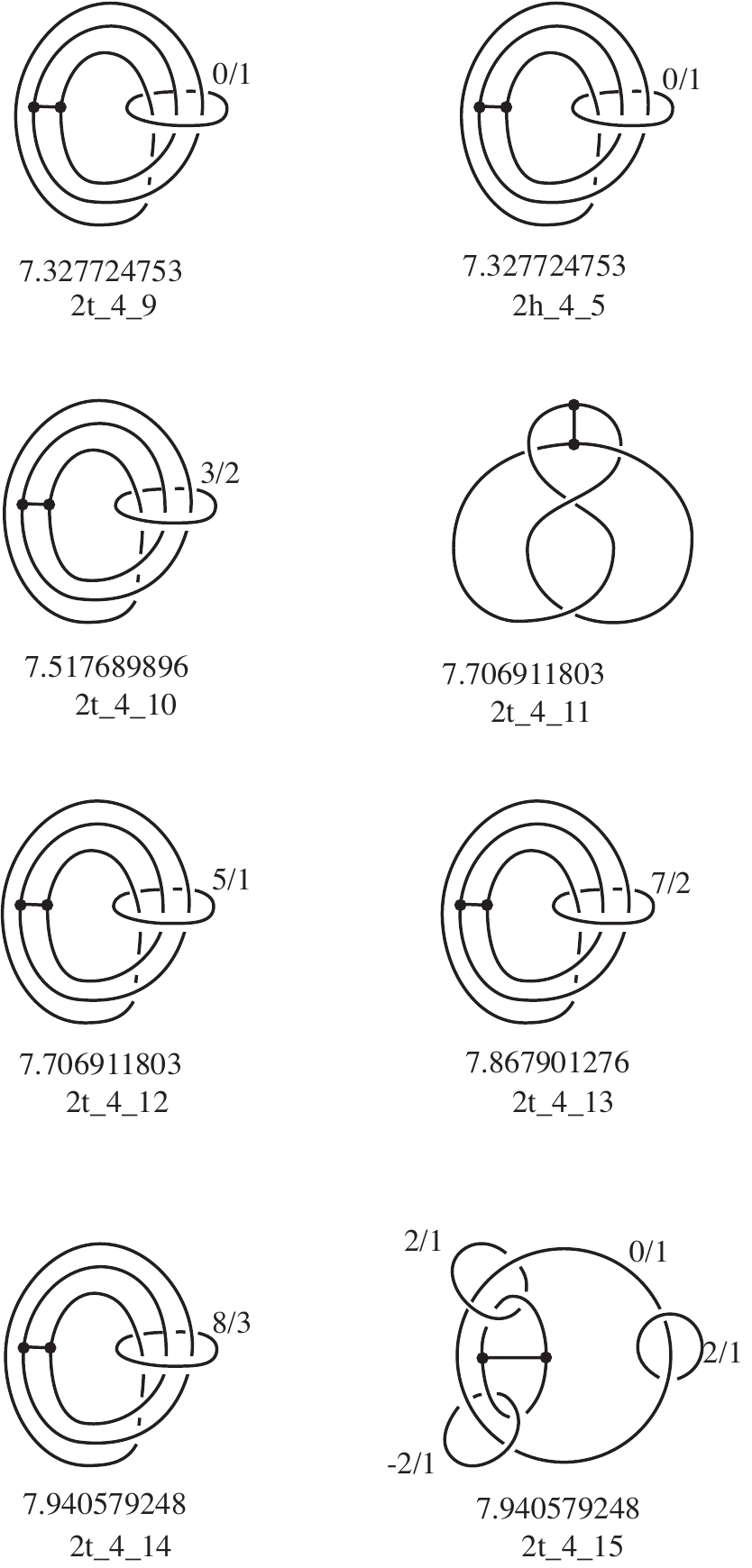}
    \mycap{\label{c=4:fig3} Complexity 4, part 3 of 4.}
    \end{center}
    \end{figure}

\ 
\vfill\eject

    \begin{figure}
    \begin{center}
    \includegraphics[scale=1.0]{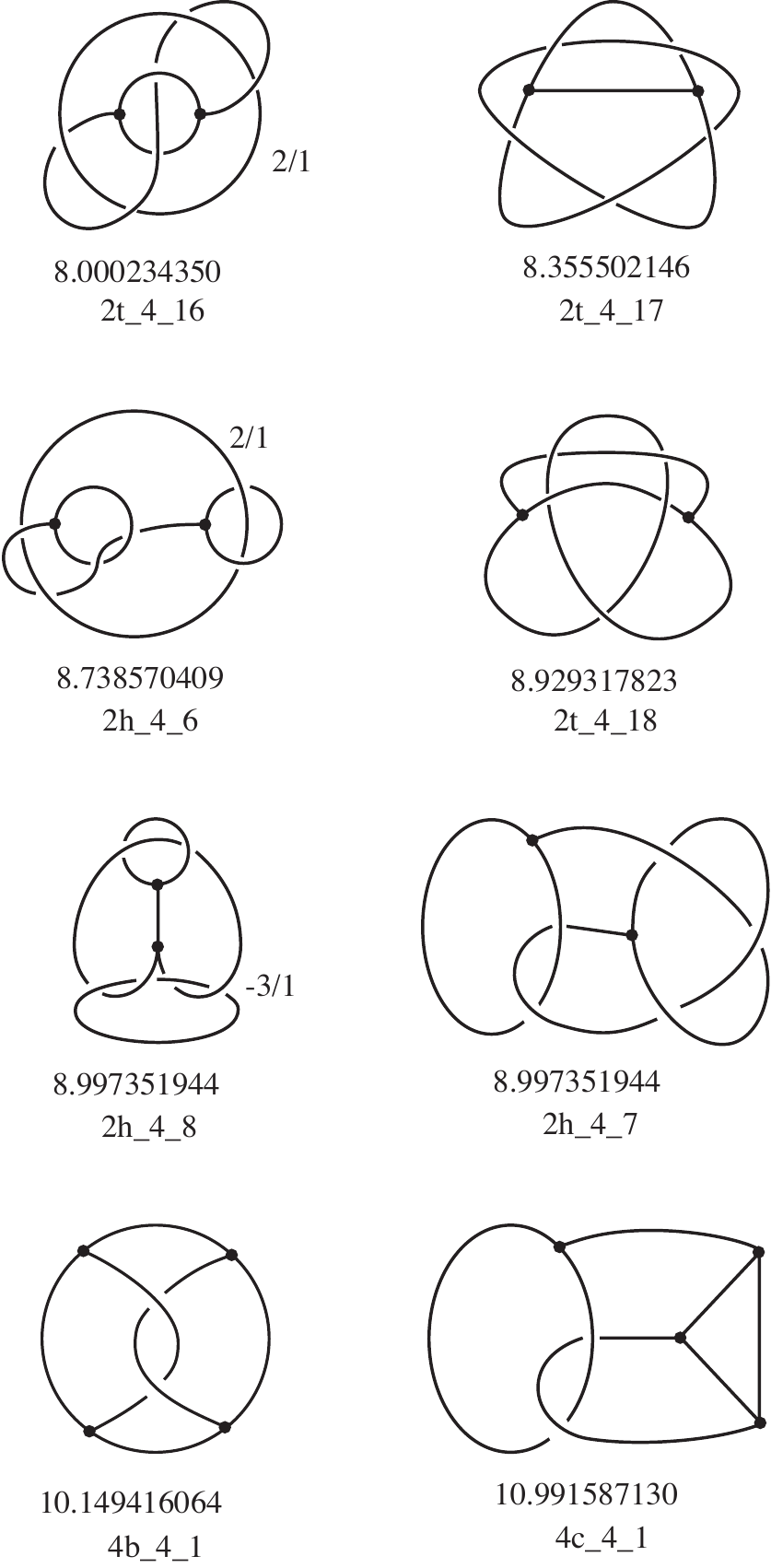}
    \mycap{\label{c=4:fig4} Complexity 4, part 4 of 4.}
    \end{center}
    \end{figure}

\ 
\vfill\eject

\ 
\vfill\eject

\end{document}